\documentclass[a4paper,reqno,10pt]{amsart}

\usepackage{amsmath,amsfonts,amssymb,amsthm,epsfig,tabls,graphicx}
\usepackage[pdftex]{color}
\usepackage{ifpdf}
\usepackage{psfrag}
\usepackage{graphicx}
\usepackage{pdfsync}

\ifpdf
\fi

\newtheorem{teo}{Theorem}[section]
\newtheorem{prop}[teo]{Proposition}

\newtheorem{lemma}[teo]{Lemma}
\newtheorem{corol}[teo]{Corollary}

\newtheorem{oss}[teo]{Remark}
\newtheorem{conj}[teo]{Conjecture}
\newtheorem*{teoS}{Iterative Selection Principle}

\newenvironment{osser}{\begin{oss}\rm }{\end{oss}}

\newcommand{\compact}{\subset\subset}    
\newcommand{\de}{\partial}

\newcommand{\eps}{\varepsilon}
\newcommand{\e}{\varepsilon}

\newcommand{\optcen}{{\mathcal Z}}

\newcommand{\mdiv}{\mathop{\mathrm{div}}}

\newcommand{\graph}{\mathop{\mathrm{gr}}}

\newcommand{\lip}{\mathop{\mathrm{Lip}}}
\newcommand{\esssup}{\mathop{\mathrm{ess~sup}}}
\newcommand{\difsim}{\bigtriangleup}

\newcommand{\carat}[1]{\chi_{#1}}

\newcommand{\Per}{{P}}

\newcommand{\R}{{\mathbb R}}

\newcommand{\N}{{\mathbb N}}

\newcommand{\cF}{{\mathcal F}}

\newcommand{\cP}{{\mathcal P}}
\newcommand{\Se}{{\mathcal S}}

\newcommand{\Hau}{{\mathcal H}}

\newcommand{\asym}{{\alpha}}
\newcommand{\defP}{{\delta P}}
\newcommand{\pillow}{{\mathcal P}}
\newcommand{\QM}{{\mathcal{QM}}}

\newcommand{\Qm}{{Q^{(m)}}}
\newcommand{\Qk}{{Q^{(k)}}}
\newcommand{\Qmmu}{{Q^{(m-1)}}}
\newcommand{\Qmj}{{Q^{(m)}_{j}}}

\newcommand{\amj}{\alpha^{(m)}_j}

\title{Best constants for the isoperimetric inequality in quantitative form}

\author[M.~Cicalese]{Marco Cicalese}
\address{Dipartimento di Matematica e Applicazioni ``R. Caccioppoli'', Universit{\`a} degli Studi di Napoli \mbox{``Federico II''}, Via Cintia, 
Monte S. Angelo, I-80126 Napoli, Italy and Institut f\"ur Angewandte Mathematik,
Universit\"at Bonn, Endenicher Allee 60, 53115 Bonn, Germany}
\email{cicalese@unina.it}

\author[G.P.~Leonardi]{Gian Paolo Leonardi}
\address{Dipartimento di Matematica Pura e Applicata ``G. Vitali'', Universit{\`a} degli Studi di Modena e Reggio Emilia, Via Campi 213/b, I-41100
    Modena, Italy}
\email{gianpaolo.leonardi@unimore.it}


\begin{document}
\begin{abstract}
We prove existence and regularity of minimizers for a class of
functionals defined on Borel sets in $\R^n$. Combining these results with a
refinement of the selection principle introduced in \cite{CicLeo10},
we describe a method suitable for the determination of the best
constants in the quantitative isoperimetric inequality with higher order
terms. Then, applying Bonnesen's annular symmetrization in a very
elementary way, we
show that, for $n=2$, the above-mentioned constants 
can be explicitly computed through a one-parameter family of convex sets
known as \textit{ovals}. This proves a further extension of a
conjecture posed by Hall in \cite{Hall92}. 
\end{abstract}

\maketitle

\tableofcontents

\section{Introduction}
Given $n\geq 2$, let $\Se^n$ be the collection of all Borel sets $E\subset\R^n$ with positive and finite Lebesgue measure $|E|$. Denoting by $B_E$ the open ball centered at $0$ with the same measure as $E$ and by $P(E)$ the perimeter of $E$ in the sense of De Giorgi, the {\it isoperimetric deficit} and the {\it Fraenkel asymmetry index} of $E\in\Se^n$ respectively read as 
$$
\defP(E)=\frac{P(E)-P(B_E)}{P(B_E)}
$$ 
and  
\begin{equation}\label{intro:asym}
\asym(E) = \inf\left\{ \frac{|E\difsim (x+B_E)|}{|B_E|},\ x\in\R^n\right\},
\end{equation}
where, as usual, $V\difsim W$ denotes the symmetric difference of the two sets $V$
and $W$.\\
The {\it sharp quantitative isoperimetric inequality} can be stated as follows: there exists a constant $C=C(n)>0$ such that 
\begin{equation}\label{intro:QII}
\defP(E)\geq C\asym(E)^2.
\end{equation}
Since the first proof of the sharp quantitative isoperimetric inequality by Fusco, Maggi and Pratelli in \cite{FusMagPra08} (see also \cite{FigMagPra10} and \cite{CicLeo10} for different proofs), a great effort has been done in order
to prove quantitative versions of several analytic-geometric inequalities (see for instance \cite{FusMagPra07}, \cite{FusMagPra09}, 
\cite{CiaFusMagPra09}, \cite{CiaFusMagPra10}, \cite{FusMagPra10}, \cite{FusMilMor10} and also \cite{Maggi08} for a survey on this argument). However, some relevant issues - such as the determination of the {\it best constant} in \eqref{intro:QII}, that is of 
\begin{eqnarray}\label{intro:bc}
C_{best}:=\max\{C>0:\, \defP(E)\geq C\asym(E)^2,\ \forall\, E\in\Se^n\},
\end{eqnarray}
the regularity of the optimal set $E_{best}$, that is of the
set such that $C_{best} = \frac{\defP(E_{best})}{\asym(E_{best})^2}$, as well as the
shape of such a set - have not yet been considered in their full generality. They seem to be challenging problems 
and only few results are known. This is basically due to the presence
of the Fraenkel asymmetry index which makes \eqref{intro:bc} a
non-local problem. As a consequence, \eqref{intro:bc} is difficult to
be tackled via standard arguments of Calculus of
Variations and shape optimization. Only in dimension $n=2$, but within
the class of convex sets, the minimizers of the isoperimetric deficit
(i.e., of the perimeter) at a fixed asymmetry index are explicitly
known. Indeed, in $1992$ Campi proved (\cite{Campi92}, Theorem $4$)
the following, equivalent statement that, among all
convex sets $E\in\Se^2$ with fixed area and perimeter $P(E)=\sigma$,
there exists a unique set $E_{\sigma}$ that maximizes the Fraenkel
asymmetry. Such a result obviously entails existence and uniqueness in
\eqref{intro:bc} restricted to convex sets. It moreover implies that the optimal convex set $E_{conv}$
agrees with $E_\sigma$ for a suitable $\sigma$. By exploiting a
symmetrization technique due to Bonnesen (\cite{Bonnesen29}), and
also known as {\it annular symmetrization}, Campi completely
characterized the set $E_\sigma$ and found an explicit  threshold
$\sigma_0$ such that, depending on whether $\sigma$ is above or below
$\sigma_0$, $E_\sigma$ is either what he called an \textit{oval}, or a
\textit{biscuit}. Here, following Campi's definition,
and assuming 
without loss of generality that the Fraenkel asymmetry of $E$ is
realized at $x=0$ (that is, $B_E$ is an \textit{optimal ball} for $E$
in the sense that $\asym(E) = \frac{|E\difsim B_E|}{|B_E|}$) we call
oval a set whose boundary is composed by
two pairs of equal and opposite circular arcs, with endpoints on
$\partial B_E$ and with common tangent lines at each point, while we
call a biscuit a set which is obtained by capping a rectangle with two
half disks (see Figure \ref{fig:ovalbisc}). 
\begin{figure}[ht]
  \centering
  \includegraphics[scale=.9]{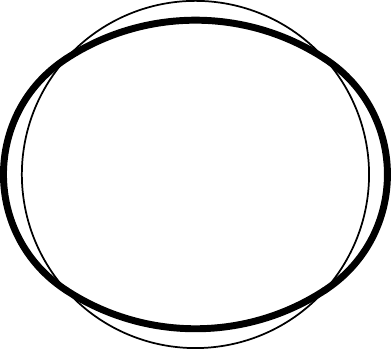}\qquad
  \includegraphics[scale=.9]{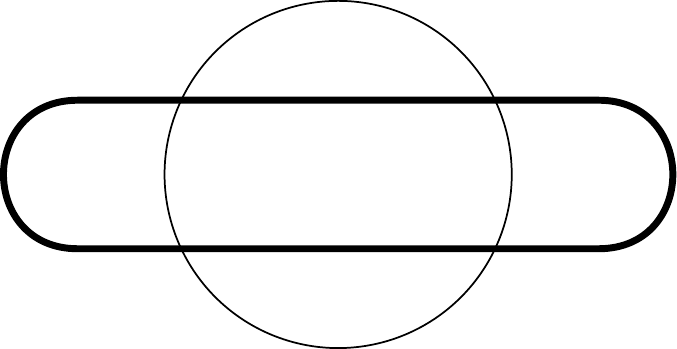}
  \caption{An oval and a biscuit, together with their optimal balls}
  \label{fig:ovalbisc}
\end{figure}
In the recent paper \cite{AlvFerNit11}, 
the authors, besides proving Campi's result in a slightly different way, 
optimize the quotient 
$\frac{\defP(E_\sigma)}{\asym(E_\sigma)^2}$ to find that
$C_{conv} = \min\limits_\sigma
\frac{\defP(E_\sigma)}{\asym(E_\sigma)^2} \simeq 0.405585$ and that
$E_{conv}$ is a biscuit. However, it is 
worth noting that, in dimension $n=2$, the problem \eqref{intro:bc} is
not solved by a convex set. An example of a non-convex
set $E_{nc}$ for which it holds  
\begin{eqnarray*}
\frac{\defP(E_{nc})}{\asym(E_{nc})^2}\simeq 0.39314
\end{eqnarray*}
is provided by the \textit{mask}, i.e. by
a set with two 
orthogonal axes of symmetry and with only two optimal balls, whose boundary is made by $8$ suitable circular arcs (see Figure \ref{fig:mask}). In the forthcoming paper
\cite{CicLeo11} it will be proved that such a set realizes the best
constant within a quite rich sub-class of planar
sets. Therefore, it seems reasonable to conjecture that the mask is
optimal with respect to all sets in $\R^2$.
\begin{figure}[ht]
  \centering
  \includegraphics[scale=.9]{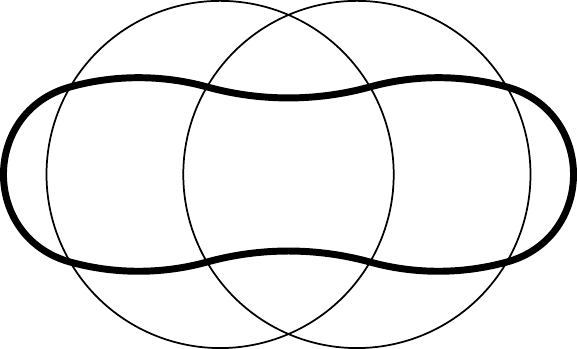}
  \caption{The mask, with its two optimal balls}
  \label{fig:mask}
\end{figure}
Up to our knowledge, and besides the two-dimensional case, problem
\eqref{intro:bc} has not been investigated. We address it here in the
first part of this paper. To this end, given $f,g:[0,2]\to \R$ two
Lipschitz-continuous functions with $g(t)$ nonnegative and zero if and only if $t=0$, for all $E\in \Se^n$ we define the functional
\[
\cF_{f,g}(E) = 
\begin{cases}\displaystyle
\frac{\defP(E) + f(\asym(E))}{g(\asym(E))}&\text{ if $\asym(E)>0$}\cr
\inf\{\liminf_h \cF_{f,g}(E_h):\ \asym(E_h)>0,\ |E_h\difsim
B|\to_h 0\}&\text{ otherwise}
\end{cases}
\]
and, for all $\alpha_0>0$ we consider the minimum problem
\begin{eqnarray}\label{intro:genmin}
\min\{\cF_{f,g}(E),\ E\in\Se^n:\ \asym(E)\geq\alpha_0\}.
\end{eqnarray}
In Theorem \ref{teo:genexists} we prove that \eqref{intro:genmin} has
a solution, while in Theorem \ref{teo:lmin} and Theorem \ref{teo:reg}
we prove that the minima are actually $\Lambda$-minimizers of the
perimeter (see Section \ref{NP} for the proper definition). 
As a consequence, on recalling classical results in the
regularity theory for quasiminimizers of the perimeter (see Theorem
\ref{teo:lminreg}), these minima are of class $C^{1,\gamma}$ for all
$\gamma<1$ (and of class $C^{1,1}$ in dimension $n=2$). Note that, by
choosing $f=0$ and $g(t)=t^2$, we have that
$\cF_{f,g}(E)=\frac{\defP(E)}{\asym(E)^2}$, hence the existence and
regularity statements hold in particular for problem 
\eqref{intro:bc}. Beside its own interest, the analysis of the more
general class of functional $\cF_{f,g}$ is here a preliminary step
towards the solution of a refinement of a problem posed in
\cite{Hall92} by Hall. In that paper, Hall conjectured that the inequality
\begin{eqnarray}\label{intro:Hall}
\defP(E)\geq \frac{\pi}{8(4-\pi)} \asym(E)^2+ o(\asym(E)^2),
\end{eqnarray}
is valid for any set $E\in\Se^2$ and that $\frac{\pi}{8(4-\pi)}$ is optimal. 
This inequality has been first proved for convex sets by Hall, Hayman
and Weitsman
in \cite{HalHayWei91,HalHay93}, and then extended by the authors to
the general case in 
\cite{CicLeo10}. It is worth pointing out that \eqref{intro:Hall} is
strongly connected with (and, actually, it is an easy consequence of)
the explicit determination of the \textit{minimizers} of the perimeter
at a fixed (small) asymmetry index. By Campi's result, we know that
minimizers among convex sets with small asymmetry are necessarily
ovals. With this information in the convex, $2$-dimensional case, it
is possible to prove not only 
\eqref{intro:Hall} but also a whole family of lower bounds of the
isoperimetric deficit by some polynomial in the asymmetry, plus
higher-order terms (see Remark 2.1 in \cite{AlvFerNit11}). 

In this direction our main contribution is Corollary \ref{cor:ovalok},
where we prove that, as soon as there
exist coefficients $c_1,\dots,c_m$ such that the estimate 
\begin{eqnarray}\label{intro:taylor}
\defP(E)\geq \sum\limits_{k=1}^{m}c_k\asym(E)^{k} + o(\asym(E)^{m})
\end{eqnarray}
is valid whenever $E$ is an oval, then \eqref{intro:taylor} is
automatically valid for any set $E\in \Se^2$. In other words, in
$\R^2$ it is not 
restrictive to only consider ovals that approximate the ball, in order
to determine the coefficients $c_k$ in \eqref{intro:taylor}. With the
aim of finding the optimal coefficients $c_k$ for  
\eqref{intro:taylor} in any dimension $n$, we introduce
the following 
family of functionals: for any $E\in \Se^n$ we define
\[
Q^{(1)}(E)=
\begin{cases}\displaystyle
\frac{\defP(E)}{\asym(E)},&\text{ if $\asym(E)>0$}\cr
\inf\{\liminf_h Q^{(1)}(E_h):\ \asym(E_h)>0,\ |E_h\difsim
B|\to_h 0\}&\text{ otherwise}
\end{cases}
\]
and, for a given integer $m\geq 2$ and assuming that 
$\Qmmu(B)\in \R$, we set 
\begin{eqnarray*}
\Qm(E)=\begin{cases}\frac{\Qmmu(E)-\Qmmu(B)}{\asym(E)}&\text{ if $\asym(E)>0$}\cr
\inf\{\liminf_h Q^{(m)}(E_h):\ \asym(E_h)>0,\ |E_h\difsim
B|\to_h 0\}&\text{ otherwise.}
\end{cases}
\end{eqnarray*}
It turns out that $c_k=\Qk(B)$, so that the problem of finding the
optimal coefficients in 
\eqref{intro:taylor} is reduced to the computation of $\Qk(B)$. We
first observe that, for 
$m\geq 2$ and $\Qk(B)\in \R$ for all $k=1,\dots,m-1$, we can
equivalently write $\Qm=\cF_{f_m,g_m}$ by choosing $f_m(\asym) =
Q^{(1)}(B)\asym + \dots + \Qmmu(B)\asym^{m-1}$ and
$g_m(\asym)=\asym^m$. Then we can combine the existence and regularity
results proved for the functionals $\cF_{f,g}$ with a penalization
technique analogous to the one exploited in \cite{CicLeo10}, to derive
the following result: 
\begin{teoS}
Let $m\geq 2$ and assume that $\Qk(B) \in \R$ for all
$k=1,\dots,m-1$. Then, there exists a sequence of sets
$(E^{(m)}_j)_j\subset \Se^n$, such that 
\begin{itemize}
\item[(i)] $|E^{(m)}_j| = |B|$, $\asym(E^{(m)}_j)>0$ and $\asym(E^{(m)}_j) \to 0$ as $j\to \infty$;

\item[(ii)] $Q^{(m)}(E^{(m)}_j) \to Q^{(m)}(B)$ as $j\to \infty$;

\item[(iii)] for each $j$ there exists a function $u^{(m)}_j\in C^1(\de B)$
  such that 
\[
\de E^{(m)}_j = \{(1+u^{(m)}_j(x))x:\ x\in \de B\}
\]
and $u^{(m)}_j \to 0$ in the $C^1$-norm, as $j\to \infty$;

\item[(iv)] $\de E^{(m)}_j$ has mean curvature $H^{(m)}_j\in
  L^\infty(\de E^{(m)}_j)$ and $\|H^{(m)}_j - 1\|_{L^\infty(\de E^{(m)}_j)} \to 0$ as $j\to \infty$.
\end{itemize}
\end{teoS}
By the Iterative Selection Principle we are allowed to compute $\Qm(B) =
\lim\limits_j \Qm(E_j^{(m)})$ via sequences of sets $E_j^{(m)}$ with asymmetry index bounded away from zero, whose
boundaries $\de E_j^{(m)}$ are smoothly converging to $\de B$ and 
such that the scalar mean-curvature functions defined on $\de
E_j^{(m)}$ are uniformly converging to the
(constant) mean curvature of $\de B$. In dimension $n=2$ we can more
precisely show that, for $j$ large enough, $E_j^{(m)}$ belongs to a
very restricted class of sets, with boundary made by arcs of circle,
and whose precise 
description is given in Section \ref{2D} (see also Figure
\ref{fig:Ej2}). Thanks to 
the minimality property of $E_j^{(m)}$, and using an elementary,
convexity-preserving, Bonnesen-style annular symmetrization on that
restricted class of sets, we finally show that $E_j^{(m)}$ are
necessarily ovals converging to $B$, whence the proof of Corollary
\ref{cor:ovalok} easily follows.

\section{Notation and preliminaries}\label{NP}
Let $E\subset \R^n$ be a Borel set, with $n$-dimensional Lebesgue
measure $|E|$. Given $x\in \R^n$ and $r>0$, we denote by $B(x,r)$  
the open Euclidean ball with center $x$ and radius $r$. We also set
$B=B(0,1)$ and $\omega_n=|B|$. 
For a set $E\in\R^n$ we denote by $\carat{E}$ its characteristic function
and correspondingly define the $L^1$ (or $L^1_{\rm loc}$) convergence of
a sequence of sets $E_j$ to a limit set $E$ in terms of the $L^1$
(or $L^1_{\rm loc}$) convergence of their characteristic functions.
The \textit{perimeter} of a Borel set $E$ inside 
an open set $\Omega\subset \R^n$ is 
\[
\Per(E,\Omega) := \sup \left\{\int_E \mdiv g(x)\, dx:\ g\in
  C^1_c(\Omega;\R^n),\ |g|\leq 1\right\}. 
\]
By Gauss-Green's Theorem, this definition provides an extension of the
Euclidean, $(n-1)$-dimensional measure of a smooth (or Lipschitz)
boundary $\de E$. We will simply write $\Per(E)$ instead of
$\Per(E,\R^n)$, and we will say that $E$ is a set of finite perimeter
if $\Per(E)<\infty$. One can check that $\Per(E,\Omega)<+\infty$ if
and only if the distributional derivative
$D\carat{E}$ is a vector-valued Radon measure in $\Omega$ with finite
total variation $|D\carat{E}|(\Omega)$. By known results (see
e.g. \cite{AmbFusPal00}) one has
$D\carat{E}=\nu_E\,\Hau^{n-1}\lfloor\de^*E$ where $\Hau^{n-1}$ is the
$(n-1)$-dimensional Hausdorff measure and $\de^*E$ is the {\it reduced boundary} of $E$, i.e., the set of those points $x\in\de E$ such that the {\it generalized inner normal} $\nu_E(x)$ is defined, that is,
\begin{equation*}
\nu_E(x)=\lim\limits_{r\to 0}\frac{D\carat{E}(B(x,r))}{|D\carat{E}|(B(x,r))}\quad\text{and}\quad|\nu_E(x)|=1.
\end{equation*}


We say that a set $E\subset \R^n$ of locally finite perimeter is
a \textit{strong $\Lambda$-minimizer} of the perimeter (here, we adopt
the terminology used in \cite{Ambrosio97}) if there exists
$R>0$ such that, for all $x\in \R^n$ and $0<r<R$, and for any
\textit{compact variation} $F$ of $E$ in $B(x,r)$ (that is, such that $E\difsim
F\subset\subset B(x,r)$) one has 
\[
\Per(E,B(x,r)) \leq \Per(F,B(x,r)) + \Lambda |E\difsim F|.
\]
We shall equivalently write $E\in \QM(R,\Lambda)$ to underline the
dependence of the definition of strong $\Lambda$-minimality on the
parameters $R$ and $\Lambda$, as well as to stress that this is a
\textit{quasiminimality} statement about $E$. 
Strong $\Lambda$-minimizers and more generally
\textit{quasiminimizers} of the perimeter have been studied after the
seminal work \cite{DeGiorgi61} by De Giorgi on the regularity theory for minimal
surfaces. We also mention the
paper by Massari \cite{Massari74} on the regularity of boundaries with prescribed
mean curvature (i.e., of minimizers of the functional $\Per(E) +
\int_E h(x)\, dx$) and the 
clear, as well as general, analysis of the regularity of
quasiminimizers of the perimeter due to Tamanini
(\cite{Tamanini82, Tamanini84}) and the lecture notes \cite{Ambrosio97} by Ambrosio. 
It is worth mentioning that a further (and notable) extension of the
regularity theory for quasiminimizers in the context
of currents and varifolds is due to Almgren (\cite{Almgren76}).

In the following theorem we state three crucial
properties verified by \textit{uniform sequences} of 
$\Lambda$-minimizers that converge in $L^1_{loc}$ to some limit set
$F$. The proof of these properties can be derived from results contained
for instance in \cite{Tamanini84} and
\cite{Ambrosio97} (see also \cite{CicLeo10} for more details).
\begin{teo}\label{teo:lminreg}
Let $E_1,\dots,E_h,\dots$ belong to $\QM(R,\Lambda)$ for some fixed
$R,\Lambda>0$ and let $E_h$ 
converge to a Borel set $F$ in $L^1_{loc}(\R^n)$ as $h\to
\infty$. Then the following facts hold. 
\begin{itemize}
\item[(i)] $F\in \QM(R,\Lambda)$. Moreover, if $\de F$ is bounded then
  $\de E_h$ converges to $\de F$ in the Hausdorff 
  metric\footnote{A sequence of compact sets $K_h$ converges to a
    compact set $K$ in the Hausdorff metric iff the infimum of
    all $\eps>0$ such that $K\subset K_h+\eps B$ and $K_h\subset
    K+\eps B$ (i.e., the so-called Hausdorff distance between $K_h$
    and $K$)
  tends to $0$ as $h\to \infty$.}.

\item[(ii)] $\de^*F$ is a smooth, $(n-1)$-dimensional hypersurface of
  class $C^{1,\gamma}$ for all $\gamma\in (0,1)$ (and $C^{1,1}$ in
  dimension $n=2$), while the singular set $\de F \setminus \de^* F$
  has Hausdorff dimension $\leq n-8$.

\item[(iii)] If $\de F$ is smooth (i.e., if the singular set of $\de
  F$ is empty) then there exists 
  $h_0$ such that, for any $h\geq h_0$, $\de E_h$ has no singular
  points, and thus it is of class
  $C^{1,\gamma}$ for all $0<\gamma<1$ ($C^{1,1}$ if $n=2$). Moreover,
  if $\de F$ is compact then $\de E_h$ can be represented as 
  the normal graph of a smooth function $u_h$ defined on $\de F$ and
  such that $u_h\to 0$ in $C^1(\de F)$, as $h\to \infty$. 
\end{itemize}
\end{teo}
\medskip

In what follows we will denote by $\Se^n$ the class of Borel
subsets of $\R^n$ with positive and finite Lebesgue measure. Given  
$E\in\Se^n$, we define its \textit{isoperimetric deficit} $\defP(E)$
and its \textit{Fraenkel asymmetry} $\asym(E)$ as follows:
\begin{equation}
  \label{deficit}
  \defP(E) := \frac{\Per(E) - \Per(B_E)}{\Per(B_E)}
\end{equation}
and
\begin{equation}
  \label{asym}
  \asym(E) := \inf\left\{ \frac{|E\difsim (x+B_E)|}{|B_E|},\ x\in \R^n\right\},
\end{equation}
where $B_E$ denotes the ball centered at the origin such that $|B_E|=|E|$
and $E\difsim F$ denotes the symmetric difference of the two sets $E$
and $F$. Since both $\defP(E)$ and
$\asym(E)$ are invariant under isometries and dilations, from now on
we will set $|E|=|B|$ so that $B_E=B$. By definition, the Fraenkel asymmetry $\asym(E)$ satisfies
$\asym(E)\in [0,2)$ and it is zero if and only if $E$ coincides with
$B$ in measure-theoretic sense and up to a translation. Notice that
the infimum in \eqref{asym} is actually a minimum. 

\section{A general class of functionals}
In this section we show existence and regularity properties of
minimizers for a general class of functionals defined on sets $E\in
\Se^n$. 


\medskip
Let $f,g:[0,2]\to \R$ be two 
Lipschitz-continuous functions with $g(t)$ nonnegative and zero if and only if $t=0$. We define
the functional $\cF_{f,g}:\Se^n\to[-\infty,+\infty]$ as follows: 
\begin{eqnarray}\label{callF}
\cF_{f,g}(E) = \begin{cases}
\displaystyle\frac{\defP(E) + f(\asym(E))}{g(\asym(E))}&\text{if }\asym(E)>0\\
\inf\{\liminf_h \cF_{f,g}(E_h):\ \asym(E_h)>0,\ |E_h\difsim
B|\to_h 0\}& \text{otherwise.}
\end{cases}
\end{eqnarray}
Clearly, $\cF_{f,g}(E)$ is invariant under isometries and dilations. Note that, in what follows, we will drop the subscripts $f$ and $g$ and simply write $\cF$ instead of $\cF_{f,g}$. 

Given $0<\alpha_0<1$ we define 
\[
\Se^n_{\alpha_0}:=\{E\subset \Se^n:\ |E|=|B|,
\asym(E)\geq \alpha_0\}
\]
and, restricting $\cF$ to $\Se^n_{\alpha_0}$, we state the following theorem:
\begin{teo}[Existence of minimizers]\label{teo:genexists}
There exists $\hat E \in\Se^n_{\alpha_0}$ such that $\cF(\hat E) \leq
\cF(E)$ for all $E\in \Se^n_{\alpha_0}$.
\end{teo}
\begin{proof}
We first observe that the subclass $\Se^n_{\alpha_0}$ is closed with
respect to the $L^1$-convergence of sets. We now fix a minimizing sequence
$(E_h)_h\subset \Se^n_{\alpha_0}$ for $\cF$ and assume $|E_h|=|B|$. We can of course assume
that $\cF(E_h)\leq C<\infty$ for all $h$ and for some constant
$C>0$. Therefore, 
\[
\Per(E_h) \leq \Per(B) + C g(\asym(E_h)) - f(\asym(E)) \leq \Per(B) +
C\max g - \min f <+\infty
\]
for all $h$. As a consequence, by a well-known compactness result for
families of sets with equibounded perimeter (see for instance \cite{Giusti84}), $(E_h)_h$ is sequentially relatively
compact in $L^1_{loc}(\R^n)$ . Starting from $E_h$
we can construct a new sequence $\hat E_h$ with the property of being
a uniformly bounded minimizing sequence for $\cF$ 
converging to a limit set $\hat E$. To this end, we shall adapt to our
case an argument originally employed by Almgren in \cite{Almgren76}. 

First, by a standard concentration-compactness argument, one can prove
that there exists $\beta_0>0$ (depending only on the data of the
problem, and not on the sequence $(E_h)_h$) and $\{x^0_h\}_h\subset \R^n$, such that 
\[
|E_h \cap (x^0_h + B)| \geq \beta_0.
\]
Of course, we can assume that $x^0_h$ is an \textit{optimal center}
for $E_h$, that is, $\asym(E_h) = \frac{|E_h\difsim (x^0_h
  +B)|}{|B|}$. 
The functional $\cF$ is invariant with respect to translations, thus
the translated sequence $E^0_h := E_h - x^0_h$ is still minimizing
and, at the 
same time, verifies $|E^0_h \cap B| \geq \beta_0$. Up to
subsequences, we can assume that $E^0_h$ converges to $E^0$ in the
$L^1_{loc}$-topology. Now, we face two possibilities: either $|E^0| =
|B|$ and this would immediately imply that $E^0_h \to E^0$ in $L^1$
(in this case we are done) or $\beta_0\leq |E^0|<|B|$. The latter
possibility corresponds to a ``loss of
mass at infinity''. In order to deal with this case, we first study the
minimality of $E^0$ with respect to the perimeter. On exploiting the
same argument contained in the proof of Lemma 3.3(ii) in
\cite{CicLeo10}, $\Per(E^0)\leq \Per(F)$ for all measurable $F\subset
\R^n$ such that $|F| = |E^0|$ and $F\difsim E^0$ is compactly
contained in $\R^n\setminus B(0,R)$ for a sufficienty large $R$. As a consequence, by
well-known results on minimizers of the perimeter subject to a volume
constraint 
we infer that $E^0$ is necessarily bounded. Let us set $R_0>0$ such that
$E_0\subset B(0,R_0)$. 
Being $B$ an optimal ball for $E^0_h$, and since $E^0_h$ converges to
$E^0$ in $L^1_{loc}$, we set
\[
\gamma := \frac{2|B\setminus E^0|}{|B|} = \lim_h \asym(E^0_h).
\]
Clearly, $\alpha_0\leq \gamma < 2$. In the case $\gamma< 1$, the set
$\tilde E = E^0\cup B_0$ minimizes the functional $\cF$, where $B_0$
is a ball such that $B_0\cap B(0,R_0+2)=\emptyset$ and $|E^0\cup B_0|
= |E^0|+|B_0| = |B|$. Indeed, it holds that $\asym(\tilde E) =
\gamma$ and $\Per(\tilde E) \leq \liminf_h \Per(E^0_h)$. Otherwise, in
the case 
$\gamma\geq 1$ we proceed differently. Since we are facing a loss of
mass in the limit, $|E^0_h \setminus B(0,R_0)|\to |B|-|E^0|>0$ as $h\to
\infty$. Hence we can find $\beta_1>0$ and $\{x^1_h\}_h \subset \R^n$ such that,
as before, $|E^0_h \cap (x^1_h + B)|\geq \beta_1$. Note that 
$|x^1_h|\to +\infty$ as $h\to \infty$, otherwise by compactness we
would contradict the inclusion $E^0\subset B(0,R_0)$. We may also
assume that 
\[
x^1_h \in \mathop{\arg\max}\limits_{x\in \R^n\setminus B(0,R_0+2)}|E^0_h
\cap (x + B)|.
\]
Arguing as before, we can extract a subsequence of $E_h^0$ (that we do not relabel) such that $E^1_h := E^0_h - x^1_h$ converges to $E^1$ in
$L^1_{loc}$, as $h\to \infty$. Moreover, we let $R_1>0$ be such that $E^1\subset B(0,R_1)$. Now, we show that there exists a constant
$C>1$ depending only on the data of the problem (and not on the
minimizing sequence $E_h$) such that 
\begin{equation}
  \label{equipezzi}
  \frac 1C |E^0|\leq |E^1|\leq C|E^0|.
\end{equation}
To prove \eqref{equipezzi} it is enough to show the first inequality,
i.e. $|E^0|\leq C|E^1|$ for some uniform $C>1$ (the other is implied
by the estimate $|E^0|\geq \beta_0$ shown above). Indeed, let us
assume $|E^1|\leq |E^0|$ (otherwise there is nothing to prove). We
consider the following modified sequence: 
\[
\tilde E_h = (\lambda_h E^0)\cup [E^0_h \setminus
(B(0,2R_0) \cup B(x^1_h,R_1))], 
\]
where $\lambda_h>1$ is such that $|\tilde E_h| = |E_h|=|B|$. Therefore,
$\lambda_h\to \left(1 + \frac{|E^1|}{|E^0|}\right)^{\frac 1n}$ as
$h\to \infty$, and by Bernoulli's inequality we also have 
\begin{equation}\label{lah}
1\leq \lambda_h \leq 1 + \frac{|E^1|}{n|E^0|} + \e_h,
\end{equation}
where $\e_h\to 0$ as $h\to \infty$. Since $|E^1|<|E^0|$, we can assume
without loss of generality that $1\leq \lambda_h < 2$ for all $h$. 
We now set $\alpha_h = \asym(E_h)$, $\tilde\alpha_h =
\asym(\tilde E_h)$, and $\delta_h = |\tilde\alpha_h -
\alpha_h|$. In what follows, to simplify notation and not to overburden the reader, we let 
``n.t.'' stand for $O(\e_h)+o\left(\frac{|E^1|}{|E^0|}\right)$. 
We first show that 
\begin{equation}
  \label{eq:dh}
  \delta_h \leq 2\frac{|E^1|}{|E^0|} + \text{ n.t.}
\end{equation}
Indeed, we recall that here
\begin{eqnarray}\label{teo:reg_alpha_lb}
\alpha_h = \frac{2|B\setminus E^0_h|}{|B|}\to \gamma\geq 1
\end{eqnarray}
as $h\to \infty$. Then, since $\lambda_h\geq 1$ we get
\begin{eqnarray*}
\alpha_h &=& \frac{2|B\setminus E^0|}{|B|} + \text{n.t.} 
= \frac{2\lambda_h^{-n}}{|B|}|\lambda_h B \setminus \lambda_h E^0| +
\text{n.t.}
\geq  \frac{2\lambda_h^{-n}}{|B|}|B \setminus \lambda_h E^0| +
\text{n.t.}\\ 
&\geq & \lambda_h^{-n}\tilde\alpha_h + \text{n.t.}
\end{eqnarray*}
whence 
\begin{eqnarray*}
\tilde\alpha_h &\leq & \lambda_h^n \alpha_h + \text{n.t.}
\leq  \alpha_h +\frac{|E^1|}{|E^0|} \alpha_h + \text{n.t.}\\
&\leq & \alpha_h +2\frac{|E^1|}{|E^0|} + \text{n.t.}
\end{eqnarray*}
Therefore we have shown that
\begin{equation}
  \label{stdelta1}
  \tilde\alpha_h \leq \alpha_h + 2\frac{|E^1|}{|E^0|} + \text{n.t.}
\end{equation}
We now consider the following two alternative cases.
\medskip

\noindent
\textit{Case $1$: there exists an optimal ball $\tilde B_h$ for
  $\tilde E_h$, such that $\tilde B_h\cap B(0,2R_0) = \emptyset$.} In this
case we have 
\begin{equation}\label{stdelta2}
\tilde\alpha_h\ =\ 2\frac{|\tilde B_h \setminus \tilde
  E_h|}{|B|}\ \geq\ 2\frac{|\tilde B_h\setminus E^0_h|}{|B|} 
\ \geq\ \alpha_h.
\end{equation}

\noindent
\textit{Case $2$: any optimal ball $\tilde B_h$ for
  $\tilde E_h$ verifies $\tilde B_h\cap B(0,2R_0) \neq \emptyset$.} In
this case, if we set $\tilde B_h = B(\tilde x_h,1)$ and recall that
$B$ is an optimal ball for $E_h$, for all $h$, we obtain 
\begin{eqnarray*}
\tilde \alpha_h &=& 2\frac{|\tilde B_h\setminus \lambda_h E^0|}{|B|} +
\text{ n.t.}
= 2\frac{\lambda_h^n}{|B|} \left|\frac{\tilde B_h}{\lambda_h}\setminus
  E^0 \right| + \text{ n.t.}\\ 
&\geq & 2\frac{\lambda_h^n}{|B|} \left[|B(\tilde
  x_h/\lambda_h,1)\setminus E^0| - |B|(1-\lambda_h^{-n})\right]+
\text{ n.t.}\\ 
&\geq & \lambda_h^n \alpha_h - 2(\lambda_h^n - 1) + \text{ n.t.}
= \alpha_h + (\alpha_h -2)(\lambda_h^n -1) + \text{ n.t.}\\
&\geq & \alpha_h - 2\frac{|E^1|}{|E^0|} + \text{ n.t.}
\end{eqnarray*}
and this proves the inequality 
\begin{equation}
  \label{stdelta3}
\tilde \alpha_h \geq \alpha_h - 2\frac{|E^1|}{|E^0|} + \text{ n.t.}
\end{equation}
which combined with \eqref{stdelta2} and \eqref{stdelta1} gives
\eqref{eq:dh}. 

\noindent
Assume by contradiction that
the ratio $\frac{|E^1|}{|E^0|}$ is not bounded below by a 
positive constant that depends only on the data of the problem. Then
by \eqref{eq:dh} and \eqref{teo:reg_alpha_lb} we have that $\tilde\alpha_h \geq \alpha_0$. Therefore, $\tilde E_h$ belongs
to the class $\Se^n_{\alpha_0}$, so that we are allowed to
compare $\cF(\tilde E_h)$ with $\cF(E^0_h)$. Thanks to the hypothesis
on $g$, we also have 
that, for $h$ large enough, $g(\alpha_h) -\lip(g)\delta_h>0$. Thus by
\eqref{eq:dh}, up to a suitable choice of the radii $R_0$ and $R_1$ (for details on this point we refer to
the proof of Lemma $3.3$ $(ii)$ in \cite{CicLeo10}), we obtain
\begin{eqnarray*}
\cF(\tilde E_h) &=& \frac{\defP(\tilde E_h) +
  f(\tilde\alpha_h)}{g(\tilde\alpha_h)}\\  
&\leq & \frac{\Per(\lambda_h E^0) + \Per(\tilde E_h\setminus
  B(0,2R_0)) -\Per(B) + \Per(B)[f(\alpha_h) +
  \lip(f)\delta_h]}{\Per(B) [g(\alpha_h) - \lip(g)\delta_h]}\\
&\leq & \frac{(\lambda_h^{n-1}-1)\Per(E^0) + \Per(E^0_h) -
  \Per(B) - \Per(E^0_h\cap B(x^1_h,R_1)) + \Per(B)[f(\alpha_h) +
  \lip(f)\delta_h]}{\Per(B)[g(\alpha_h) - \lip(g)\delta_h]}+o(1).
\end{eqnarray*}
Since
\[
\liminf_h\Per(E^0_h\cap B(x^1_h,R_1)) \geq \Per(E^1),
\]
by exploiting the isoperimetric inequality we get
\begin{equation}\label{lestimate}
\cF(\tilde E_h) \leq \frac{(\lambda_h^{n-1}-1)\Per(E^0) +
  \Per(E^0_h)-\Per(B) - n\omega_n^{\frac 1n}|E^1|^{\frac{n-1}{n}} +
\Per(B)[f(\alpha_h) + 
  \lip(f)\delta_h]}{\Per(B)[g(\alpha_h) - \lip(g)\delta_h]}+o(1),
\end{equation}
Then, combining
\eqref{lah} and \eqref{lestimate}, and after 
some straightforward computations, we get 
\begin{equation*}
\cF(\tilde E_h) \leq \cF(E^0_h) - 
  \frac{\omega_n^{\frac 1n-1}}{\max g}|E^1|^{\frac{n-1}{n}}+o(1),
\end{equation*}
which contradicts the optimality of the sequence
$E^0_h$, thus proving \eqref{equipezzi}.

Since the volume of $E_h$ equals $|B|$, an immediate consequence of
\eqref{equipezzi} is that there exists a finite family
$\{E^0,E^1,\dots,E^N\}$ of sets of finite perimeter, obtained
as $L^1_{loc}$-limits of suitably translated subsequences of the
initial minimizing sequence $E_h$, which satisfy
\begin{itemize}
\item[(a)] $|E^0|+\dots +|E^N| = |B|$;

\item[(b)] $\liminf_h \Per(E_h) \geq \sum_{i=0}^N \Per(E^i)$;

\item[(c)] $\asym(E_h) \to \min\limits_{i=0\dots N} \min\limits_{x\in \R^n}
  \frac{2|(x+B)\setminus E^i|}{|B|}$ as $h\to \infty$.
\end{itemize}
The proof of (a) and (b) is routine. On the other hand, (c) follows
from the fact that given $i,j\in \{1,\dots,N\}$ with $i\neq j$, the two
sets $E^i$ and $E^j$ are respectively obtained as limits in
$L^1_{loc}$ of a subsequence of $(E_h)_h$, up to suitable translation vectors
$x^i_h$ and $x^j_h$ that satisfy $\lim_h |x^i_h - x^j_h|=+\infty$. \\
We can now construct a minimizer of $\cF$ by simply setting
\[
\hat E = E^0 \cup (v + E^1) \cup (2v + E^2)\cup \dots \cup (Nv + E^n),
\]
where $v\in \R^n$ is any vector such that $E^i \subset B(0,|v|/2-2)$ for all
$i=0,\dots,N$. In this way, we guarantee by (a) above that $|\hat E| =
|B|$ and, by (c), that
\begin{equation}\label{asymcont}
\asym(\hat E) = \min\limits_{i=0\dots N} \min\limits_{x\in \R^n}
  \frac{|(x+B)\setminus E^i|}{|B|} = \lim_h \asym(E_h).
\end{equation}
Finally, by (b) and \eqref{asymcont} we conclude that $\hat E$ is a
minimizer of $\cF$. 
\end{proof}
\medskip

In the following Lemma we recall an elementary but useful estimate
of a difference of asymmetries in terms of the volume of the symmetric
difference of the corresponding sets. 
\begin{lemma}\label{lemma:stimasym}
Let $E\in\Se^n$ with $|E|=|B|=\omega_n$. For all $x\in \R^n$ and 
for any $F\in\Se^n$ with $E\difsim F
\compact B(x,\frac 12)$, it holds that
$|\asym(E) - \asym(F)| \leq \frac{2^{n+2}}{(2^n-1)\omega_n} |E\difsim F|$. 
\end{lemma}

The next is a crucial theorem in our analysis asserting the $\Lambda$-minimality of the minimizers 
of the functional in \eqref{callF}.

\begin{teo}[$\Lambda$-minimality]\label{teo:lmin}
Let $\cF$ be the functional defined in \eqref{callF}. Then, there exists $\Lambda>0$ such that any
minimizer $E\in \Se^n$ of $\cF$, with $|E| = |B|$, is a
$\Lambda$-minimizer of the perimeter. 
\end{teo}
\begin{proof}
Of course, if $\asym(E) = 0$ there is nothing to prove, since $E$ is a
ball (and thus a well-known $\Lambda$-minimizer of the perimeter) up
to null sets. We now assume $\asym(E)>0$ and fix $x\in \R^n$ and a
compact variation $F$ of $E$ in $B(x,\frac 12)$. It is not restrictive
to assume that $\Per(F)\leq \Per(E)$ and that $\asym(F)>0$. Since $\cF(E) \leq \cF(F)$, we have
\begin{equation}
  \label{eq:1}
(\defP(E) + f(\asym(E)))g(\asym(F)) \leq g(\asym(E)) (\defP(F) + f(\asym(F))).
\end{equation}
Then, combining Lemma \ref{lemma:stimasym} and the Lipschitz
continuity of $g$, we have
\begin{equation}
  \label{eq:5}
|g(\asym(E)) - g(\asym(F))| \leq \lip(g)|\asym(E) - \asym(F)| \leq
C_{n,g}|E\difsim F|,
\end{equation}
with $C_{n,g} = \lip(g)\frac{2^{n+2}}{(2^n-1)\omega_n}$. We now set 
\[
\chi = \text{ sign of }\ (\defP(E) + f(\asym(E)))
= \text{ sign of }\ \cF(E)
\]
and observe that by \eqref{eq:5} 
\begin{equation}
  \label{eq:2}
(\defP(E) + f(\asym(E)))(g(\asym(E)) - \chi C_{n,g}|E\difsim F|) \leq
(\defP(E) + f(\asym(E)))g(\asym(F)),
\end{equation}
thus plugging \eqref{eq:2} into \eqref{eq:1} and dividing by
$g(\asym(E))$ we get
\begin{equation}
  \label{eq:3}
\defP(E) - \defP(F) \leq f(\asym(F)) - f(\asym(E)) +
C_{n,g}|\cF(E)|\cdot |E\difsim F|.
\end{equation}
On recalling that $f$ is Lipschitz, we have \eqref{eq:5} with $f$
replacing $g$, and therefore we obtain
\begin{equation}
  \label{eq:4}
\defP(E) - \defP(F) \leq (C_{n,f} + C_{n,g}|\cF(E)|)\cdot |E\difsim F|.
\end{equation}
Then, we note that 
\begin{equation}
  \label{eq:6}
\defP(E) - \defP(F) = \frac{\Per(E) -
  \frac{\Per(B)}{\Per(B_F)}\Per(F)}{\Per(B)}   
\end{equation}
and that
\begin{eqnarray}\label{eq:7}\nonumber
\frac{\Per(B)}{\Per(B_F)} &=&
\left(\frac{|E|}{|F|}\right)^{\frac{n-1}{n}}= \left(1+ 
\frac{|E|-|F|}{|F|}\right)^{\frac{n-1}{n}}\\  
&\leq & \left(1+ \frac{|E\difsim F|}{|F|}\right)^{\frac{n-1}{n}}\\ \nonumber 
&\leq & 1+ \frac{n-1}{n}\cdot \frac{|E\difsim F|}{|F|}\\ \nonumber
&\leq & 1+ \frac{4(n-1)}{3n|B|}\cdot |E\difsim F|,
\end{eqnarray}
where we have used Bernoulli inequality and the fact that $E\difsim F
\subset\subset B(x,\frac 12)$ implies $|F| \geq |B| - |B(x,\frac 12)| =
\frac 34 |B|$. Finally, using \eqref{eq:6} and \eqref{eq:7} 
we can rewrite \eqref{eq:4} as
\[\Per(E)\leq \Per(F)+|E\difsim F|\left(	C_{n,f}+C_{n,g}|\cF(E)|\Per(B)+\frac{4(n-1)}{3n|B|}\Per(E)\right),
\]
which turns out to imply 
\[
\Per(E) \leq \Per(F) + \Lambda|E\difsim F|,
\]
once we note that the constant
\begin{eqnarray}\label{Lambda}
\Lambda=C_{n,f}+C_{n,g}|\cF(E)|\Per(B)+\frac{4(n-1)}{3n|B|}\Per(E)
\end{eqnarray}
depends only on the 
dimension $n$ and on the functions $f$ and $g$. 
\end{proof}
As a consequence of Theorems \ref{teo:genexists}, \ref{teo:lmin}
  and \ref{teo:lminreg} one obtains the following
\begin{teo}[Regularity]\label{teo:reg}
Let $\cF$ be the functional defined in \eqref{callF} and let
$E\in\Se^n$ be a minimizer of $\cF$, with $|E| = |B|$. Then, $\de^* E$
is of class $C^{1,\eta}$ for any $\eta\in (0,1)$ ($C^{1,1}$ for
$n=2$), while the singular set $\de E \setminus \de^* E$ has Hausdorff
dimension $\leq n-8$. 
\end{teo}
In the following lemma, we let $E$ be a
minimizer of $\cF$ and we show that the (scalar) mean curvature $H$
of $\de E$ belongs to $L^\infty(\de E)$. Moreover, we
compute a first 
variation inequality of $\cF$ at $E$ that translates into a
quantitative estimate of the oscillation of the mean curvature.
\begin{lemma}\label{lemma:firstvar}
Let $E$ be a minimizer of $\cF$. Then $\de^* E$ has scalar mean
curvature $H \in L^\infty(\de^* E)$ (with
orientation induced by the inner normal to $E$). Moreover, for
${\Hau}^{n-1}$-a.e. $x,y\in \de^* E$, one has 
\begin{equation}
  \label{Ffirstvar}
  |H(x) - H(y)| \leq \frac{n}{n-1}\Big(|\cF(E)|\lip(g)+\lip(f)\Big);
\end{equation}
\end{lemma}
\begin{proof}
To prove the theorem we consider a ``parametric
inflation-deflation'', that will lead to the first variation
inequality \eqref{Ffirstvar}. 

Let us fix $x_1,\ x_2\in\partial^* E$ be such that $x_1\not=x_2$. By
Theorem \ref{teo:reg}, there exist $r>0$ such that, for $m=1,2$
\[
\de E\cap B(x_m,r) = \de^* E\cap B(x_m,r)
\]
is the graph of a smooth function $f_m$ defined on an open set
$A_m\subset \R^{n-1}$, with respect to a suitable reference
frame, so that the set $E\cap B(x_m,r)$ ``lies below'' the graph of
$f_m$. For $m=1,2$ we take $\varphi_m\in C^1_c(A_m)$ such that
$\varphi_m\geq 0$ and 
\begin{equation}
  \label{eq:fipsi}
  \int_{A_m} \varphi_m = 1.
\end{equation}
Let $\eps>0$ be such that, setting $f_{m,t}(w) = f_m(w) +(-1)^m
t\varphi_m(w)$ for $w\in A_m$, one has $\graph(f_{m,t}) \subset
B(x_m,r)$ for all $t\in (-\eps,\eps)$. We use the functions
$f_{m,t}$, $m=1,2$, to modify the set $E$, i.e. we define $E_t$ such
that $E_t\difsim E$ is compactly contained in $B(x_1,r)\cup B(x_2,r)$,
with $\de E_t\cap B(x_m,r) = \graph(f_{m,t})$ for $m=1,2$. 
By \eqref{eq:fipsi} one immediately deduces that $|E_{t}| =
|E|$. Moreover, by a standard computation one obtains
\begin{equation}\label{teo:firstvar_perimeter}
\frac{1}{n-1}\frac{d}{dt}\Per(E_{t})_{|_{t=0}} =
\int_{A_1}h_1 \varphi_1 -
\int_{A_2}h_2 \varphi_2,
\end{equation}
where for $m=1,2$
\[
h_m(v) := H(v,f_m(v)) = -\frac{1}{n-1}\mdiv \left( \frac{\nabla
f_m(v)}{\sqrt{1+|\nabla f_m(v)|^2}}\right).
\]
Then, by Theorem 4.7.4 in \cite{Ambrosio97}, the
$L^\infty$-norm of $H$ over $\de E$ turns out to be bounded by a
constant depending only on $\Lambda$ and on the dimension $n$. 

By the definition of $E_{t}$ one can verify that, for $t>0$  
\begin{eqnarray}\label{teo:firstvar_asymmetry}
|\asym(E_{t})-\asym(E)|\leq \frac{t}{\omega_n}.
\end{eqnarray}
By \eqref{teo:firstvar_perimeter} and \eqref{teo:firstvar_asymmetry},
and for $t>0$, we also have that
\begin{eqnarray*}
\cF(E_t)&=&\frac{\Per(E_t)-\Per(B)+\Per(B)f(\asym(E_t))}{\Per(B)g(\asym(E_t))}\\
&=&\frac{\Per(E)-\Per(B)+\Per(B)f(\asym(E_t))}
{\Per(B)g(\asym(E_t))}+\frac{t}{\Per(B)g(\asym(E_t))}\frac{d}{dt}\Per(E_{t})_{|_{t=0}}+o(t)
\\
&=&
\cF(E)\cdot \frac{g(\asym(E))}{g(\asym(E_t))} + \frac{f(\asym(E_t)) -
  f(\asym(E))}{g(\asym(E_t))} + 
\frac{t}{n\omega_ng(\asym(E_t))}\frac{d}{dt}\Per(E_{t})_{|_{t=0}}+o(t)\\
&\leq&\cF(E) + \frac{t}{n\omega_ng(\asym(E_t))}\left(n|\cF(E)|\lip(g)+n\,\lip(f)+\frac{d}{dt}\Per(E_{t})_{|_{t=0}}\right)+o(t)
\end{eqnarray*}
Exploiting now the minimality hypothesis $\cF(E)\leq \cF(E_{t})$
in the previous inequality, dividing by $t>0$, multiplying by
$n\omega_ng(\asym(E_t))$, and finally taking the limit as
$t$ tends to $0$, we obtain 
\begin{equation}\label{ttendeazero}
0\leq \frac{d}{dt}P(E_{t})_{|t=0} +n|\cF(E)|\lip(g)+n\,\lip(f).    
\end{equation}
Let now $w_m\in A_m$ be a Lebesgue
point for $h_{f_m}$, $m=1,2$. On choosing a sequence
$(\varphi_m^k)_k\subset C^1_c(A_m)$ of non-negative
mollifiers, such that  
\[
\lim_k \int_{A_m} h_{f_m} \varphi_m^k = h_{f_m}(w_m)
\]
for $m=1,2$, we obtain that for $E_{t}^k$ defined as before, but
with $\varphi_m^k$ replacing $\varphi_m$, it holds
\begin{eqnarray}\label{eq:firstvar_curvature}
\frac{1}{n-1}\lim_{k}\frac{d}{dt}P(E_{t}^k)_{|_{t=0}} &=& \lim_k
\int_{A_1} h_{f_1}\varphi_1^k - \int_{A_2} h_{f_2}\varphi_2^k\\ 
\nonumber &=& h_{f_1}(w_1) - h_{f_2}(w_2).
\end{eqnarray}
Moreover, from \eqref{ttendeazero} with $E_{t}^k$ in place of
$E_{t}$ and thanks to \eqref{eq:firstvar_curvature}, we get 
\begin{equation}
h_{f_2}(w_2) - h_{f_1}(w_1) = -\frac{1}{n-1}\lim_k
\frac{d}{dt}P(E_{t}^k)_{|t=0} \leq 
\frac{n}{n-1}\left(|\cF(E)|\lip(g)+\lip(f)\right).    
\end{equation}
Finally, the proof of \eqref{Ffirstvar} is achieved
by exchanging the roles of $x_1$ and $x_2$.
\end{proof}
\begin{osser}\label{oss:firstvar}
Under the hypotheses of the previous theorem, if we additionally
suppose that $f,g$ are $C^1$ functions, then arguing as above we 
obtain, for ${\Hau}^{n-1}$-a.e. $x,y\in \de^*E$,
\begin{eqnarray*}
  |H(x) - H(y)| \leq \frac{n}{n-1}\left(|\cF(E)|\cdot |g'(\asym(E))|+|f'(\asym(E))|\right).
\end{eqnarray*}
Moreover, if $x,y\in \de^* E$ are such that the inflation-deflation procedure in a
small neighbourhood of $\{x,y\}$ does not change the asymmetry (i.e.,
if $\asym(E_t) = \asym(E)$ for $t$ small) then we get $H(x) =
H(y)$. This property is verified, in particular, by any pair $(x,y)$ of
points belonging to a same \textit{free region}. We call free region
any connected component of the set
\[
\Omega = \R^n \setminus \bigcup_{x\in \optcen(E)} (x+\de B),
\]
where $\optcen(E)$ is the set of optimal centers for $E$, that is,
$z\in \optcen(E)$ if and only if $|E\difsim (z+B)| = |B| \asym(E)$.  
Note that $\Omega$ is open, since it is the complement of a compact
set. It is not difficult to show 
that small inflations-deflations localized in a free region
$A\subset \Omega$ do not
change the asymmetry, thus implying that the intersection $A\cap \de E$ has constant
mean curvature. Clearly, the value of the mean curvature can change
from one free region to another.
\end{osser}

\section{Quantitative isoperimetric quotients of order $m$}

For any $E\in \Se^n$ we set
\begin{equation*}
Q^{(1)}(E)= \begin{cases}\displaystyle
\frac{\defP(E)}{\asym(E)}&\text{if }\asym(E) >0\\
\inf\left\{\liminf_k \frac{\defP(F_k)}{\asym(F_k)},\
  |F_k|=|B|,\ \asym(F_k)>0,\ |F_k\difsim B|\to_k 0 \right\} &\text{otherwise. }
\end{cases}
\end{equation*}
We recall that the optimal power of the asymmetry in the quantitative
isoperimetric inequality is $2$, thus we necessarily have $Q^{(1)}(B)
= 0$ (this can be also seen through a straightforward computation made on a sequence of ellipsoids converging to the ball $B$). 
 
Analogously, for a given integer $m\geq 2$ and assuming that
$Q^{(k)}(B) \in \R$ for all $k=1,2,\dots,m-1$, we define for any $E\in \Se^n$ such that
$\asym(E)>0$ 
\begin{equation*}
\Qm(E)=\frac{\Qmmu(E)-\Qmmu(B)}{\asym(E)},
\end{equation*}
and
\begin{equation*}
\Qm(B)= \inf\left\{\liminf_k \Qm(F_k),\
  |F_k|=|B|,\ \asym(F_k)>0,\ |F_k\difsim B|\to_k 0 \right\}.
\end{equation*}
Note that, assuming $\asym(E)>0$ and recalling that $Q^1(B) =
  0$, it turns out that 
\[
Q^{(2)}(E) = \frac{\defP(E)}{\asym(E)^2}
\]
is precisely the sharp quantitative isoperimetric quotient. Hence, by
\eqref{intro:bc}, it is bounded from below by a positive, dimensional
constant and, as a consequence, $Q^{(2)}(B)$ is finite and strictly
positive. 

In what follows, we shall often say that $\Qm$ is
  \textit{well-defined} simply meaning that we are inductively assuming
  $Q^{(k)}(B)$ finite for all $k=1,2,\dots,m-1$. Clearly, this does
  \textit{not} necessarily imply that also $\Qm(B)$ is finite. 
One can easily check that the functional $\Qm$ is lower
semicontinuous on the whole class $\Se^n$. However, it is not possible
to immediately get the finiteness of $\Qm(B)$, and in particular one
cannot \textit{a priori} exclude that $\Qm(B)=-\infty$. 

By the previous definition, if $m\geq 2$, a well defined $\Qm$ can be 
equivalently written, for $\asym(E)>0$, as
\begin{equation}\label{def:Qm_psi}
\Qm(E) = \frac{\defP(E) - \psi_m(\asym(E))}{\asym(E)^{m}}
\end{equation}
where we have set
\[
\psi_m(\asym) = \sum_{i=1}^{m-1} Q^{(i)}(B)\asym^i \, .
\]
We now define the penalized functionals $\Qmj$ for $m\geq 2$.
We assume $\Qmmu(B)\in \R$ and choose a recovery sequence $(W^{(m)}_j)_j$
for $\Qm(B)$. Then, setting $\amj = \asym(W^{(m)}_j)$, for any
$E\in \Se^n$ with $\asym(E)>0$ we define
\begin{equation}\label{def:Qmj}
\Qmj(E) = \frac{\Qmmu(E) - \Qmmu(B) +
  \left(\frac{\asym(E)}{\amj} - 1\right)^2}{\asym(E)}.
\end{equation}
One can immediately check that
\begin{eqnarray*}
\Qmj(E) &=& \Qm(E) + \frac{\left(\frac{\asym(E)}{\amj} - 1\right)^2}{\asym(E)}\\ 
&=& \frac{\defP(E) - \psi_m(\asym(E)) + \asym(E)^{m-1}\left(\frac{\asym(E)}{\amj} - 1\right)^2}{\asym(E)^m}.
\end{eqnarray*}
Note that, for $m=2$, the definition of $Q^{(2)}_j$ is slightly different
from the penalized functional introduced in \cite{CicLeo10}. 

We conclude the section with an immediate consequence of Theorem \ref{teo:genexists} concerning the
existence of minimizers of the functional $\Qm$. 
\begin{corol}
Assume that $\Qm$ is well-defined and that
  $\Qm(B)>-\infty$. Then, $\Qm$ attains its minimum in the class $\Se^n$.
\end{corol}
\begin{proof}
Either $\Qm(B) \leq \Qm(F)$ for all $F\in \Se^n$ (and thus $B$ is the
required minimizer) or $\inf_{\Se^n} \Qm = \inf_{\Se^n_\beta} \Qm$ for
some $\beta>0$. In the latter case, we first observe that, on choosing 
$f(\alpha)=\psi_m(\alpha)$ and $g(\alpha)=\alpha^m$, we have $\cF_{f,g}=\Qm$. Then, applying
Theorem \ref{teo:genexists}, we get that $\Qm$ is minimized on $\Se^n_\beta$, whence the thesis.
\end{proof}

\section{The Iterative Selection Principle}

\begin{teo}[Iterative Selection Principle]\label{teo:ISP}
Let $m\geq 2$ and assume that $Q^{(k)}(B) \in \R$ for all
$k=1,\dots,m-1$. Then, there exists a sequence of sets $(E^{(m)}_j)_j\subset \Se^2$, such that
\begin{itemize}
\item[(i)] $|E^{(m)}_j| = |B|$, $\asym(E^{(m)}_j)>0$ and $\asym(E^{(m)}_j) \to 0$ as $j\to \infty$;

\item[(ii)] $Q^{(m)}(E^{(m)}_j) \to Q^{(m)}(B)$ as $j\to \infty$;

\item[(iii)] for each $j$ there exists a function $u^{(m)}_j\in C^1(\de B)$
  such that 
\[
\de E^{(m)}_j = \{(1+u^{(m)}_j(x))x:\ x\in \de B\}
\]
and $u^{(m)}_j \to 0$ in the $C^1$-norm, as $j\to \infty$;

\item[(iv)] $\de E^{(m)}_j$ has curvature $H^{(m)}_j\in
  L^\infty(\de E^{(m)}_j)$ and $\|H^{(m)}_j - 1\|_{L^\infty(\de E^{(m)}_j)} \to 0$ as $j\to \infty$.
\end{itemize}
\end{teo}
Note that in the iterative selection principle we do not assume the
finiteness of $\Qm(B)$. On one hand, the case $\Qm(B) = +\infty$ is
trivial since the thesis of the theorem is satisfied by any
sufficiently nice sequence of sets with positive asymmetry and
converging to $B$ (for instance, by a sequence of ellipsoids). On the
other hand, the case $\Qm(B) = -\infty$ can interestingly enough be treated
the same way as the finite case.

The proof of Theorem \ref{teo:ISP} will require some intermediate
results. Here we follow more or less the same proof scheme adopted in
\cite{CicLeo10}. First, we make the following observation:
\begin{lemma}[Ball exclusion]\label{lemma:noball}
Assume $\Qm$ well-defined. If $(F_h)_h$ is a minimizing sequence for
$\Qmj$, then there exists 
$\beta>0$ and $h_0\in \N$ such that $\asym(F_h) \geq \beta$ for all
$h\geq h_0$ (in other words, $F_h$ cannot converge to the ball $B$).
\end{lemma}
\begin{proof}
By contradiction, assume $\asym(F_h)\to 0$ as $h\to \infty$ (up to
subsequences). By the very definition of $\Qmmu(B)$, thanks to its finiteness, we have that
\begin{equation*}
\Qmmu(B) \leq \Qmmu(F_h)+o(1).
\end{equation*}
As a consequence, it holds that 
\begin{equation}\label{pippa}
\defP(F_h) - \psi_{m-1}(\asym(F_h)) \geq \Qmmu(B)
\asym(F_h)^{m-1} + o(\asym(F_h)^{m-1}).
\end{equation}
On the other hand, we have $\psi_m(\asym) = \psi_{m-1}(\asym) +
\Qmmu(B) \asym^{m-1}$, therefore thanks to \eqref{pippa}, and owing to the definition of $\Qmj$, we obtain
\begin{eqnarray*}
\Qmj(F_h) &\geq & \frac{
  \asym(F_h)^{m-1}\left(\frac{\asym(F_h)}{\amj}-1\right)^2+o(\asym(F_h)^{m-1})}{\asym(F_h)^m}\\
& = & \frac{\asym(F_h)^{m-1} + o(\asym(F_h)^{m-1})}{\asym(F_h)^m}.
\end{eqnarray*}
Since the right-hand side of this inequality tends to $+\infty$ as $h$
diverges, while the functional $\Qmj$ is not identically $+\infty$, we
get a contradiction. 
\end{proof}
On combining Theorem \ref{teo:genexists} with Lemma \ref{lemma:noball} we can prove the following
\begin{prop}
Assume $\Qm$ well-defined. Then, the associated penalized functional $\Qmj$
admits a minimizer $E^{(m)}_j\in \Se^n$ with $\asym(E^{(m)}_j)>0$. 
\end{prop}
\begin{proof}
We first observe that, on choosing
$f_j(\alpha)=\psi_m(\alpha)+\alpha^{m-1}\left(\frac{\alpha}{\amj}-1\right)^2$
and $g(\alpha)=\alpha^m$, we have $\cF_{f_j,g}=\Qmj$. Then, the thesis
is a direct consequence of Theorem \ref{teo:genexists} and of Lemma \ref{lemma:noball}. 
\end{proof}

\begin{lemma}\label{lemma:lowerbound}
 Assume $\Qm$ well-defined and $\Qm(B)>-\infty$. Then, there exists $\lambda_m\in \R$ such that 
\[
\Qm(E) \geq \lambda_m
\]
for all $E\in \Se^n$. 
\end{lemma}
\begin{proof}
We argue by contradiction. If there existed a sequence $(F_h)_h$ of sets
in $\Se^n$ satisfying $\Qm(F_h) \to -\infty$ as $h\to \infty$, by
the assumption on $\Qm(B)$ we would find $\beta>0$ such that
$\asym(F_h)\geq \beta$ for $h$ sufficiently large. Consequently, from
the very definition of $\Qm$ and the fact that $\defP(F_h)\geq 0$ we would deduce that
\[
\Qm(F_h) \geq \frac{-\sup\{\psi_m(\alpha),\ \beta\leq\alpha<2\}}{\asym(F_h)} \geq
-\frac{|\sup\{\psi_m(\alpha),\ \beta\leq\alpha<2\}|}{\beta},
\]
which leads to a contradiction on observing that, by the assumptions, $\sup\{\psi_m(\alpha),\ \beta\leq\alpha<2\}\in\R$.
\end{proof}

The next proposition deals with the asymptotic behavior, as $j\to +\infty$, of the sequences $(E_j^{(m)})_j$, $(\Qm(E_j^{(m)}))_j$ and $(\asym(E_j^{(m)}))_j$.

\begin{lemma}\label{lemma:asymm_to_zero}
Let $\Qm$ be well-defined and let $E_j^{(m)}$ be a minimizer
of $\Qmj$, with $\Qm(B)<+\infty$. Then 
$E_j^{(m)}\to B$ in $L^1$, $\Qm(E_j^{(m)})\to \Qm(B)$ and $\frac{\asym(E_j^{(m)})}{\amj}\to 1$.
\end{lemma}
\begin{proof}
Since $\Qmj(W_j^{(m)})=\Qm(W_j^{(m)})\to Q^{(m)}(B)<+\infty$,
we can suppose that there exists a constant $\Lambda_m>0$ such that $\Qmj(W_j^{(m)})\leq
  \Lambda_m$ for all $j$. Therefore, we have also $\Qmj(E_j^{(m)})\leq
\Qmj(W_j^{(m)})\leq \Lambda_m$ for all $j$. Again using the definition of $\Qmj$ we get that
\begin{equation}\label{stQ0}
\Lambda_m \geq \Qmj(E_j^{(m)}) = \frac{\Qmmu(E_j^{(m)}) - \Qmmu(B) + \left(\frac{\asym(E_j^{(m)})}{\amj}-1\right)^2}{\asym(E_j^{(m)})},
\end{equation} 
whence by Lemma \ref{lemma:lowerbound} applied to $\Qmmu$ we obtain
\begin{equation}\label{stQ1}
\Lambda_m \geq \frac{\lambda_{m-1} - \Qmmu(B) + \left(\frac{\asym(E_j^{(m)})}{\amj}-1\right)^2}{\asym(E_j^{(m)})}.
\end{equation} 
From \eqref{stQ1} and thanks to the trivial estimate
$\asym(E_j^{(m)})< 2$ we get for all $j$
\[
\left(\frac{\asym(E_j^{(m)})}{\amj}-1\right)^2 < 2\Lambda_m -
\lambda_{m-1} + \Qmmu(B),
\]
which means that $\left(\frac{\asym(E_j^{(m)})}{\amj}-1\right)^2$ is
uniformly bounded. Since $\amj\to 0$ we immediately infer that
$\asym(E_j^{(m)})\to 0$, as $j$ diverges. But then the sequence
$(E_j^{(m)})_j$ converges to $B$ in $L^1$ and we have by definition of $\Qmmu(B)$
\begin{equation}
  \label{Qmmueps}
\Qmmu(E_j^{(m)}) \geq \Qmmu(B) +o(1).
\end{equation}
Therefore, plugging
\eqref{Qmmueps} into \eqref{stQ0} we obtain after simple calculations
\begin{equation}\label{stQ2}
\left(\frac{\asym(E_j^{(m)})}{\amj}-1\right)^2\leq
\asym(E_j^{(m)})\Lambda_m + o(1) \to 0\qquad \text{as }j\to \infty.
\end{equation} 
We have proved that $E_j^{(m)}\to B$ in $L^1$ and that
$\frac{\asym(E_j^{(m)})}{\amj}\to 1$, as $j$
diverges. The remaining claim follows directly from the definition of
$\Qm(B)$ and from the inequalities
\[
\Qm(E_j^{(m)}) \leq \Qmj(E_j^{(m)}) \leq \Qm(W_j^{(m)}).
\]
\end{proof}

We now state a lemma about the $\Lambda$-minimality and
the regularity of minimizers of $\Qmj$, as $j\to \infty$. 
\begin{lemma}[Regularity]\label{lemma:regconv}
Let $\Qm$ be well-defined and let $\Qm(B)<+\infty$. Then there exists $j_1\in
\N$ such that, for all $j\geq j_1$ and for any minimizer $E^{(m)}_j$
of $\Qmj$, we have that 
\begin{itemize}
\item[(i)] $E^{(m)}_j$ is a $\Lambda$-minimizer of the perimeter, with
  $\Lambda$ uniform in $j$;  

\item[(ii)] $\de E^{(m)}_j$ is of class $C^{1,\eta}$ for any $\eta\in
  (0,1)$; 

\item[(iii)] $\de E^{(m)}_j$ converges to $\de B$ in the $C^1$-topology, as
  $j\to \infty$.  

\end{itemize}
\end{lemma}
\begin{proof}
A first attempt to prove (i) could be to directly apply Theorem
\ref{teo:lmin}. In this way, we would prove that $E := E^{(m)}_j$ is
a $\Lambda$-minimizer of the perimeter, but we would also obtain
$\Lambda = \Lambda_j$ dependent on $j$, and this dependence may
degenerate in the (not \textit{a priori} excluded) case $\Qm(B)
= -\infty$.
Therefore, in order to show that $\Lambda$ does
not depend on $j$ we have to deal with the limit case $\Qm(B) =
-\infty$ and, for that, we need a slight refinement of the
computations already performed in the proof of Theorem \ref{teo:lmin}. 
In the following, we assume $m\geq 3$ (the case $m\leq 2$ is treated in
\cite{CicLeo10}). We let $\cF = \Qmj$, that is we set 
\[
f(\alpha) = -\psi_m(\alpha) +
\alpha^{m-1}\left(\frac{\alpha}{\amj}-1\right)^2 
\]
and
\[
g(\alpha) = \alpha^m
\]
in the definition of $\cF = \cF_{f,g}$. Then 
we fix a point $x\in \R^n$ and a compact variation $F$ of $E$
inside $B(x,\frac 12)$. We distinguish the following two cases.
\medskip

\noindent 
In the \textit{first case}, we suppose that 
\begin{equation}
  \label{eq:stimadifsim1}
  |E\difsim F|> \asym(E).
\end{equation}
Being $\cF(E)$ uniformly bounded from above by some constant
$C>0$, and thanks to \eqref{eq:stimadifsim1}, we obtain 
\begin{eqnarray}\label{eq:defasym}
\nonumber \defP(E) &\leq & \psi_{m-1}(\asym(E)) +
\asym(E)^{m-1}(|\Qmmu(B)|+2C)\\ 
&\leq & \tilde C \asym(E)\\ 
\nonumber &< & \tilde C |E\difsim F|
\end{eqnarray}
with $\tilde C$ depending only on $\Qmmu(B), C$ and 
$\psi_{m-1}$. Now, from \eqref{eq:defasym} and by the isoperimetric
inequality in $\R^n$ we derive
\begin{eqnarray}
  \label{eq:LC_1}
\nonumber \Per(E) &\leq & \Per(B) + \Per(B) \tilde C |E\difsim F|\\ 
&\leq & \Per(F) + \Per(B) - \Per(B_F) + C_0|E\difsim F|\\ 
\nonumber &\leq & \Per(F) + C_1|E\difsim F|,
\end{eqnarray}
where the last inequality follows from Bernoulli's inequality, with a
constant $C_1$ that does not depend on $j$.

\medskip

\noindent
In the \textit{second case}, we suppose on the contrary that 
\begin{equation}
  \label{eq:stimadifsim2}
  |E\difsim F|\leq \asym(E)
\end{equation}
and observe that the constant $\Lambda$ arising in the proof of
Theorem \ref{teo:lmin} can be estimated in a more precise way. Indeed, 
since by Lemma \ref{lemma:asymm_to_zero} we have $\Qm(E_j^{(m)})\to\Qm(B)$ as $j\to +\infty$, by the very definition of $\Qm$ and the hypothesis 
$\Qm(B)<+\infty$, we obtain that $\Per(E)$ is bounded by a
dimensional constant and, on recalling \eqref{Lambda}, we get
\begin{eqnarray*}
\Lambda &=&
C_{n,f}+C_{n,g}|\cF(E)|\Per(B)+\frac{4(n-1)}{3n|B|}\Per(E)\\ 
&\leq & C_{n}\left(1+\lip(f) + \lip(g)\cdot |\cF(E)|\right),
\end{eqnarray*}
where $C_n$ is a positive, dimensional constant. 
Observe now that $\lip(g)$ can be replaced by $g'(\asym(E)) =
m\asym(E)^{m-1}$ up to possibly taking a larger constant $C_n$. In fact \eqref{eq:stimadifsim2}, together with Lemma
\ref{lemma:stimasym} and the monotonicities
of $g(\alpha)$ and of $g'(\alpha)$, implies 
\begin{eqnarray*}
|g(\asym(F)) - g(\asym(E))| &\leq & |g(\asym(E)+ c_n|E\difsim F|)
-g(\asym(E))|\\ 
&\leq & c_n g'((1+c_n)\asym(E)) |E\difsim F|\\ 
&=& c_n(1+c_n)^{m-1}\cdot m\asym(E)^{m-1}\cdot |E\difsim F|,
\end{eqnarray*}
with $c_n = \frac{2^{n+2}}{(2^n-1)\omega_n}$. In conclusion, we get
\begin{equation}\label{eq:stimaL}
\Lambda \leq C_{n}(1+\lip(f) + m\asym(E)^{m-1}\cdot |\cF(E)|).
\end{equation}
Now, to show that $\Lambda$ is uniformly bounded in $j$ we only need
to estimate the product 
\[
\asym(E)^{m-1}\cdot |\cF(E)| = \asym(E_j^{(m)})^{m-1} \cdot |\Qmj(E_j^{(m)})|.
\]
We first observe that the assumption $\Qm(B) < +\infty$ implies that
\[
\lim_j \Qmmu(E_j^{(m)}) = \Qmmu(B).
\]
Then, we obtain
the desired estimate by writing $\Qmj(E_j^{(m)})$ in terms of $\Qmmu$ and
recalling that $m-2>0$:
\begin{eqnarray*}
\asym(E)^{m-1} \cdot |\Qmj(E)| &=& \asym(E)^{m-2}\left(\Qmmu(E) -
  \Qmmu(B) + \left(\frac{\asym(E)}{\amj}-1\right)^2\right)\\ 
&\leq &C \asym(E)^{m-2}.
\end{eqnarray*}
Appealing again to Lemma \ref{lemma:asymm_to_zero} we have that $\asym(E_j^{(m)})\to 0$, which, by the estimate above, implies
\[
\lim_j \asym(E_j^{(m)})^{m-1} \cdot |\Qmj(E_j^{(m)})| = 0.
\]
As a result, in this case $\Lambda = \Lambda_j \leq C_2$ for some dimensional
constant $C_2>0$. Thanks to this last estimate and to \eqref{eq:LC_1}, we
conclude that 
\[
\Lambda = \Lambda_j \leq \max(C_1,C_2)
\]
holds, which completes the proof of (i).
\medskip

\noindent
Finally, to prove (ii) and (iii) one can follow the same argument contained
in the proof of Lemma 3.6 in \cite{CicLeo10}. 
\end{proof}

Applying Lemma \ref{lemma:firstvar} and Remark \ref{oss:firstvar}, in
the following proposition we explicitly write the first variation inequality of $\Qmj$ at $E^{(m)}_j$. Regarding the latter 
as a quantitative estimate of the oscillation of the mean curvature of $\partial E_j^{(m)}$, we deduce its limit as $j\to \infty$.

\begin{lemma}\label{lemma:firstvarQ}
Let $\Qm$ be well-defined,
  $\Qm(B)<+\infty$ and $j_1$ as in Lemma \ref{lemma:regconv}.  If $E^{(m)}_j$ minimizes $\Qmj$ then, for all $j\geq j_1$ it holds
\begin{itemize}
\item[(i)] $\de E^{(m)}_j$ has scalar mean curvature $H^{(m)}_j \in
  L^\infty(\partial E^{(m)}_j)$ (with
orientation induced by the inner normal to $E^{(m)}_j$, and
  with $L^\infty$-norm bounded by a constant independent of
  $j$). Moreover, for ${\Hau}^{n-1}$-a.e. $x,y\in \de E^{(m)}_j$, one has 
\begin{equation}
  \label{firstvarQmj}
  |H^{(m)}_j(x) - H^{(m)}_j(y)| \leq \frac{n}{n-1}\,\Delta_j^{(m)}(\asym(E_j^{(m)})),
\end{equation}
where 
$$
\Delta_j^{(m)}(\alpha)=m\alpha^{m-1}|\Qmj(E_j^{(m)})|+|\psi'_m(\alpha)|+
(m-1)\alpha^{m-2}\left(\frac{\alpha}{\amj}-1\right)^2 +
2\frac{\alpha^{m-1}}{\amj}\left|\frac{\alpha}{\amj} -1\right|; 
$$
\item[(ii)] $\lim_j\|H^{(m)}_j-1\|_{L^\infty(\de E^{(m)}_j)}=0$.
\end{itemize}
\end{lemma}
\begin{proof}
Let us note that, on choosing $f_j(\alpha)=\psi_m(\alpha)+\alpha^{m-1}\left(\frac{\alpha}{\amj}-1\right)^2$ and $g(\alpha)=\alpha^m$, we have 
$\Qmj=\cF_{f_j,g}$. As a consequence, the proof of $(i)$ easily follows from Lemma \ref{lemma:firstvar} and Remark \ref{oss:firstvar}, by explicitly computing the first derivatives of $f_j$ and $g$. 
Next we point out that, by Lemma
\ref{lemma:firstvar},  
 $\|H^{(m)}_j\|_{L^{\infty}(\partial E_j^{(m)})}\leq
 4\Lambda/(n-1)$. Thanks to Lemma \ref{lemma:asymm_to_zero} and to the
 definition of $\psi_m$, recalling that $Q^{(1)}(B)=0$ we also have that
 $\lim_j\Delta_j^{(m)}(\asym(E_j^{(m)}))=0$, which implies 
\begin{eqnarray}\label{eq:firstvar_osc}
\lim_j\ \esssup_{x,y\in\partial E_j}|H_j(x)-H_j(y)|= 0.
\end{eqnarray}
From this observation and arguing exactly as in \cite{CicLeo10} Lemma
3.7, one can easily complete the proof of (ii).
\end{proof}
\medskip

We finally obtain the proof of the Iterative Selection
Principle. 
\begin{proof}[proof of Theorem \ref{teo:ISP}]
Statements $(i)$ and $(ii)$ follows by Lemma
\ref{lemma:asymm_to_zero}. The proof of statement $(iii)$ is an
elementary consequence of Lemma \ref{lemma:regconv}, while $(iv)$ follows by Lemma \ref{lemma:firstvarQ}.
\end{proof}

\section{Optimal asymptotic lower bounds for the deficit: the $2$-dimensional case}
\label{2D}
As we have seen in the previous section, the Iterative Selection
Principle allows us to set up a recursive procedure for the
computation, for any fixed integer
$m$, of the optimal constants $c_i = Q^{(i)}(B)$ for $i=1,\dots,m$,
such that the estimate 
\[
\defP(E) \geq \sum_{i=1}^m c_i \asym(E)^i + o(\asym(E)^m)
\]
holds true for any set $E\in \Se^n$. We recall that, in any dimension
$n$, $c_1 = Q^{(1)}(B) = 0$ and $0<c_2 = Q^{(2)}(B)<+\infty$. 
The main result of this section is the following:
\begin{teo}\label{teo:THC1}
We have 
\[
c_m = \Qm(B) = \lim_j \Qm(E^{(m)}_j),
\]
where in dimension $n=2$ and for $j$ large enough, $E^{(m)}_j$ is 
an oval, i.e. a member of a one-parameter family of
$2$-symmetric, convex deformations of the disk, with boundary of
class $C^1$ and formed by two pairs of congruent arcs of circle. 
\end{teo}
One can see a picture of an \textit{oval} in Figure \ref{fig:Ej2}. 
Since the isoperimetric deficit and the asymmetry of an oval
can be explicitly computed, we obtain Corollary \ref{cor:ovalok}
below, that is a generalization of Theorem 4.6 in \cite{CicLeo10} and
thus of previous results obtained for convex sets by Hall, Hayman and
Weitsman in \cite{HalHayWei91, HalHay93, Hall92}, by
Campi in \cite{Campi92} and by Alvino, Ferone and Nitsch in \cite{AlvFerNit11}.  
\begin{corol}\label{cor:ovalok}
Assume that the estimate
\[
\defP(E) \geq \sum_{i=2}^m c_i \asym(E)^i + o(\asym(E)^m)
\]
is valid for ovals. Then, it is valid for all measurable sets in $\R^2$.
\end{corol}
The proof of Corollary \ref{cor:ovalok} is an immediate consequence of
Theorem \ref{teo:THC1}. 
\medskip

\noindent
Following \cite{Bonnesen29}, we now introduce a tool that will be used in the proof of Theorem \ref{teo:THC1}. Given $E\in\Se^2$, fix a line $l$ and a point $x$ on $l$. For any $r>0$, consider $\partial B(x,r)$ and let 
$\lambda(r)=P(\partial B(x,r)\cap E)$. On $\partial B(x,r)$ take two opposite arcs, each of length $\frac{\lambda(r)}{2}$, so that $l$ passes through the 
midpoint of both arcs. The set obtained as the collection of all such arcs, when $r$ varies in $(0,+\infty)$ is called the {\it Bonnesen annular symmetrized set} of $E$ and in what follows it will be denoted by $E^{as}$. 
As an elementary property of the annular symmetrization, we have
that for all $r>0$,  
\begin{eqnarray*}
|E\cap B(x,r)|=|E^{as}\cap B(x,r)|,
\end{eqnarray*}
which in particular implies $|E|=|E^{as}|$. Another relevant, though
elementary, property of the annular symmetrization is the following
\begin{teo}[Bonnesen, 1924]
Let $E$ be a convex set and let $r\leq R$ be, respectively, the inner
and outer radius of the annulus $C_{r,R}(x)$ centered in $x$,
containing $\de E$, and having minimal width $R-r$. Then, if $E^{as}$
is an annular symmetrization of $E$ centered at $x$ with respect to
some line through $x$, one has $\Per(E^{as}) \leq \Per(E)$ with
equality if and only if $E^{as} = E$.
\end{teo}
The proof of this theorem is not completely elementary, as one must
show that if $x$, $r$ and $R$ are the parameters defining the optimal
annulus $C_{r,R}(x)$, then both $\de B(x,r)$ and $\de B(x,R)$
intersect $\de E$ in at least two distinct points (this property is
crucial to show that the perimeter does not increase after the
symmetrization). Moreover, this symmetrization is not closed in the
class of convex sets, i.e. it does not preserve convexity in general
(see \cite{Campi92}). 
However, we shall not use Bonnesen's result but prove instead a much
more elementary property of the annular symmetrization restricted to a
special class of sets, on which it preserves area, smoothness and also
convexity, while not increasing the perimeter. To this end, for any
integer $k\geq 2$, we start defining a special class $\cP(k)$ of sets as follows: we say that
a set $E\subset\Se^2$ belongs to $\cP(k)$ if 
\begin{figure}[ht]
  \centering
  \includegraphics[scale=.9]{oval.pdf}\qquad
  \includegraphics[scale=.9]{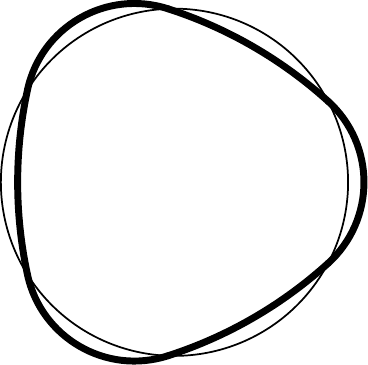}\qquad
  \includegraphics[scale=.9]{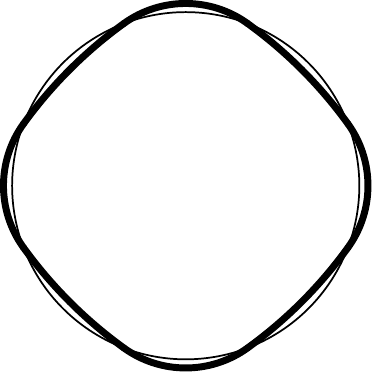}
  \caption{Three examples of
    set belonging to $\cP(2), \cP(3)$ and $\cP(4)$ (from left to
    right). The one on the left is an oval.}
  \label{fig:Ej2}
\end{figure}
\begin{itemize}
\item $|E| = |B|$;

\item $\de E$ is of class $C^1$;

\item there exist two constants $0<h_2<1< h_1$ such that, $\de
  E\setminus B$ is a union of $k$ congruent arcs of circle with
  curvature $h_1$, and similarly $\de E \cap B$ is a union of $k$ congruent
  arcs of circle with curvature $h_2$.
\end{itemize}
Moreover, the elements of the class $\cP(2)$ are called
\textit{ovals}. Some sets belonging to $\cP(k)$ for $k=2,3,4$ are depicted
in Figure \ref{fig:Ej2}.

If $E\in \cP(k)$ then, up to a rotation, its boundary $\partial E$ can be
parameterized by the angular coordinate $\theta\in[0,2\pi]$ as a curve
$\gamma_k:[0,2\pi]\to \R^2$ enjoying the following properties:
\begin{itemize}
\item[(i)] there exists $\beta\in (0,\frac{\pi}{k})$ such that,
  if $\overline\gamma_k$ denotes the restriction of $\gamma_k$
  to the interval $[-\beta,\frac{2\pi}{k}-\beta]$, one has
\begin{eqnarray}\label{THC1:curve}
\overline\gamma_k(\theta)=
\begin{cases}
C_{1}+{\frac{1}{h_1}}(\cos\theta,\sin\theta)&\text{ if }
\theta\in[-\beta,\beta)\\ 
C_{2}+{\frac{1}{h_2}}(\cos\theta,\sin\theta)&\text{ if }
\theta\in[\beta,\frac{2\pi}{k}-\beta] 
\end{cases}
\end{eqnarray}
where 
\begin{eqnarray*}
C_{1}&=&\left(\frac{1}{h_1}\cos(\arcsin(h_1\sin\beta)) -
\cos{\beta}\right)(1,0), \\ 
C_{2}&=&-\left(\frac{1}{h_2}\cos(\arcsin(h_2\sin\beta)) -
\cos{\beta}\right)\left(\cos\frac{2\pi}{k},\sin\frac{2\pi}{k}\right)
\end{eqnarray*}
are the centers of the two arcs of $\overline\gamma_k$ with curvatures
$h_1$ and $h_2$ respectively (see Figure \ref{fig:Ej}.

\item[(ii)] Denoting for any $\theta\in[0,2\pi]$ by $R(\theta)\in SO(2)$ the counterclockwise rotation of angle $\theta$ around the origin, then for all $l\in\{1,\dots,k-1\}$ 
\begin{eqnarray*}
\gamma_k\left(\theta+l\frac{2\pi}{k}\right)=R\left(l\frac{2\pi}{k}\right)\overline\gamma_k(\theta).
\end{eqnarray*}
\end{itemize}
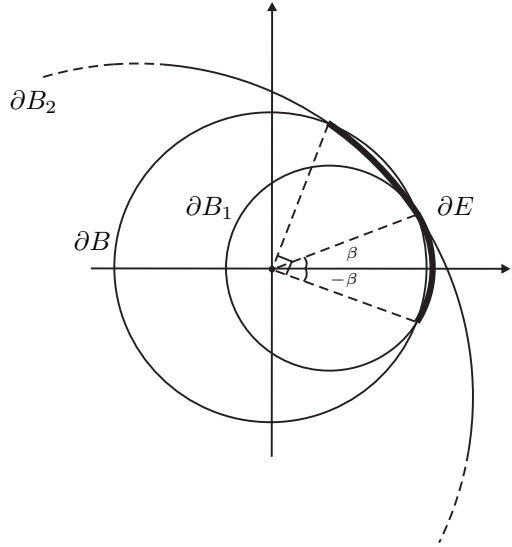
\begin{figure}[h!]
  \centering
  \includegraphics[scale=.36]{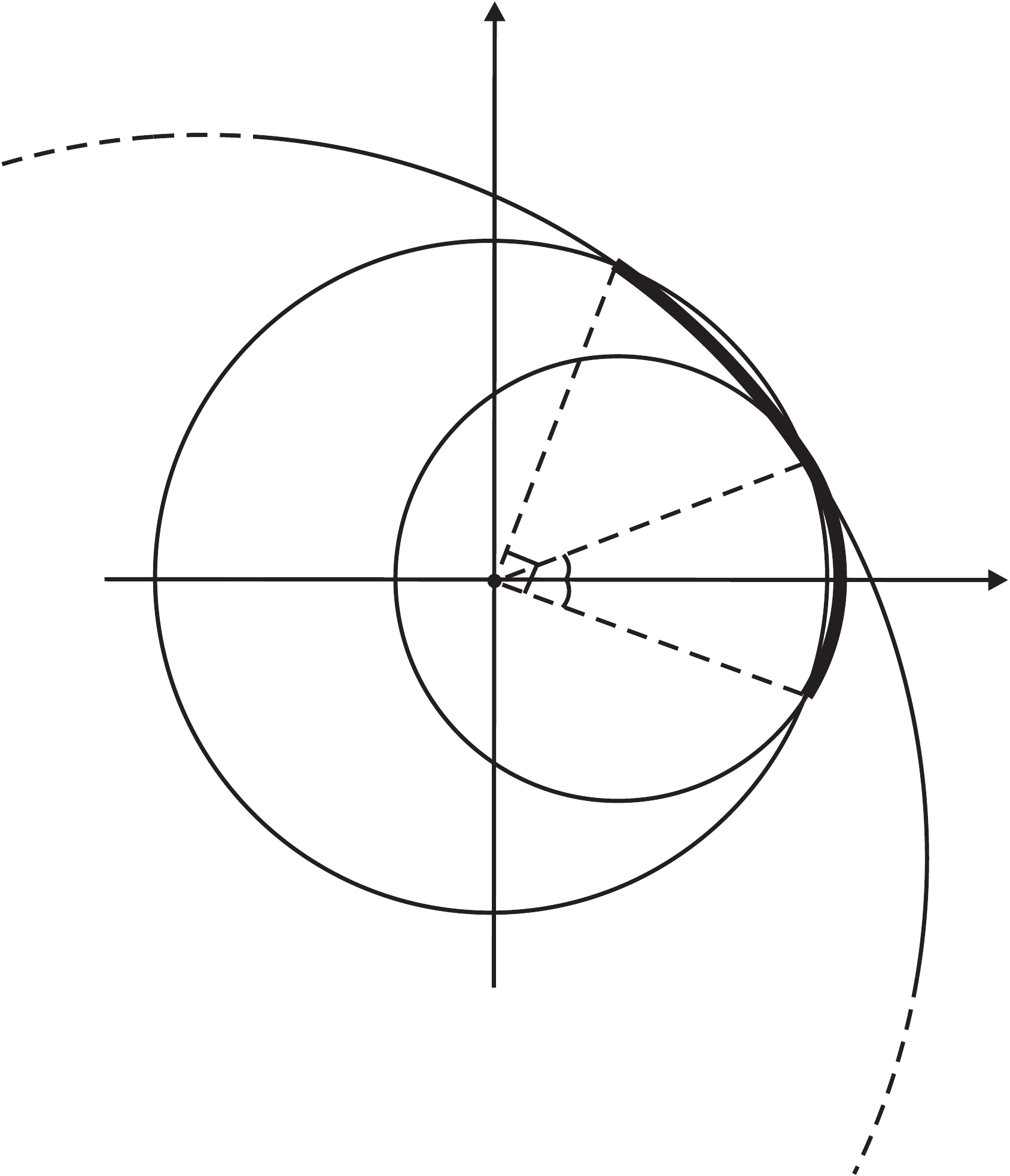} 
  \caption{Parameterization of $1/k$ of $\de E$ (bold line) for $k=4$. Here
    $B_1=B(C_1,\frac{1}{h_1})$ and
    $B_2=B(C_2,\frac{1}{h_2})$ }
  \label{fig:Ej}
  \begin{picture}(0,0)
\put(20,150,5){\tiny $-\beta$}\put(26,160,5){\tiny $\beta$}
\put(60,176,5){$\partial E$}\put(-34,176,5){$\partial B_1$}\put(-100,216,5){$\partial B_2$}
\put(-76,163,5){$\partial B$}
\end{picture}

\end{figure}

\medskip
Before proceeding with the proof of Theorem \ref{teo:THC1}, we prove a
lemma on the uniqueness of the optimal center for a strictly convex
set, valid in any dimension $n$. We recall that $x\in \R^n$ is an
optimal center for $E\in \Se^n$ if $|B_E|\asym(E) = |E\difsim
(x+B_E)|$, and that the set of all optimal centers for $E$ is denoted by
$\optcen(E)$. 
\begin{lemma}\label{lemma:unico}
Let $E$ be a strictly convex set in $\R^n$. Then the optimal center of
$E$ is unique.
\end{lemma}
\begin{proof}
Without loss of generality, we assume $B_E = B$. Arguing by
contradiction, let $x_1,\ x_2\in\optcen(E)$ with $x_1\not
=x_2$. By the very definition of Fraenkel asymmetry we have that, for $i\in\{1,2\}$,  
\begin{eqnarray}\label{THC1:optimality}
|E\cap (x_i+B)|=|B|\left(1-\frac{\asym(E)}{2}\right).
\end{eqnarray}
We now set, for $\lambda\in[0,1]$,  $x_{\lambda}=\lambda
x_1+(1-\lambda)x_2$ and we observe that 
\begin{equation*}
E\cap(x_\lambda+B)\supseteq \lambda\left(E\cap(x_1+B)\right)+(1-\lambda)\left(E\cap(x_2+B)\right).
\end{equation*} 
Since $E\cap(x_\lambda+B)$ is a convex set, we can now exploit the
Brunn-Minkowski inequality (see for example \cite{BurZal88}) together with \eqref{THC1:optimality} to get
\begin{eqnarray*}\label{THC1:BM}
|E\cap(x_\lambda+B)|^{\frac 1n}&\geq& |\lambda\left(E\cap(x_1+B)\right)+(1-\lambda)\left(E\cap(x_2+B)\right)|^{\frac 1n}\\
&\geq& \lambda|E\cap(x_1+B)|^{\frac
  1n}+(1-\lambda)|E\cap(x_2+B)|^{\frac 1n}\\
&=&|B|^{\frac 1n}\left(1-\frac{\asym(E)}{2}\right)^{\frac 1n}.
\end{eqnarray*}
The previous inequality shows that $|E\cap(x_\lambda+B)|\geq
|B|(1-\frac{\asym(E)}{2})$. By the definition of Fraenkel
asymmetry, the opposite inequality holds true as well. Thus
$|E\cap(x_\lambda+B)|= |B|(1-\frac{\asym(E)}{2})$. Such an
equality turns out to imply the equality in the Brunn-Minkowski
inequality, which is equivalent to saying that, up to translation,
the sets $E\cap (x_\lambda+B)$ are homotetic to $E\cap (x_1+B)$ for all
$\lambda\in [0,1]$. Since they all
have the same measure, they actually coincide up to translation. As a
result, we obtain the flatness of $\partial E\cap (x_\lambda+B)$ in
the direction of the vector $x_2-x_1$, but this is in 
contradiction with the strict convexity of $E$. 
\end{proof}
\medskip

\begin{proof}[Proof of Theorem \ref{teo:THC1}]
By the Iterative
Selection Principle, and assuming by induction that the coefficients
$c_i = Q^{(i)}(B)$ are finite for all $i=1,\dots,m-1$, we obtain that
$\Qm(B) = \lim\limits_j E^{(m)}_j$, with $E^{(m)}_j$ satisfying the
thesis of Theorem \ref{teo:ISP}. 

Then, again by Theorem \ref{teo:ISP} we may suppose that, for $j$
sufficiently 
large, $H_{\partial E^{(m)}_j}(x)>{\frac 12}$ for $\mathcal
H^1$-a.e. $x\in\partial E^{(m)}_j$, 
hence $E^{(m)}_j$ is a strictly convex set in $\R^2$. Therefore, by
Lemma \ref{lemma:unico} we have $\optcen(E^{(m)}_j) = \{x_0\}$ and, up
to a translation, we may assume that 
$x_0=0$. Therefore, $B$ and $\R^2\setminus B$ are \textit{free
  regions} for $E^{(m)}_j$ in the sense of Remark
\ref{oss:firstvar}. By Theorem 
\ref{teo:ISP} and Remark \ref{oss:firstvar} we finally deduce that
$E^{(m)}_j$ belongs to $\cP(k)$ for some $k\geq 2$.  

Now, by exploiting the annular symmetrization,
we will show that necessarily $E^{(m)}_j$ is an oval, that is,
$E^{(m)}_j\in \cP(2)$. Indeed, assume
by contradiction that $E^{(m)}_j \in \cP(k)$ for some $k>2$. In order
to simplify notation, we drop some indices and let $E=E^{(m)}_j$ be
such that $\partial E$ is described by $\gamma_k$ given in
\eqref{THC1:curve}. Then, we will prove that the \textit{annular
  symmetrized set} $E^{as}$ with respect to the origin of the
reference system satisfies 
\begin{eqnarray*}
|E^{as}|=|E|,\quad
\asym(E^{as})=\asym(E),\quad
P(E^{as})< P(E),
\end{eqnarray*}
which would give the desired contradiction with the minimality of
$E$. To this aim, let $\rho_1, \rho_2>0$ be such that
$B(0,\rho_1)\subset E\subset
B(0,\rho_2)$ and that both $\partial B(0,\rho_1)$ and $\partial
B(0,\rho_2)$ are tangent to $\partial E$. In other words, the annulus
$C_{\rho_1,\rho_2} = B(0,\rho_2) \setminus B(0,\rho_1)$ is the one of
minimal thickness among those containing $\de E$. Let us set $S$ as the
circular sector which is given in polar coordinates $(r,\theta)$ by
$S=[0,+\infty)\times[0,\frac{\pi}{k}]$. Given now $\rho\in(\rho_1,\rho_2)$ we
set $A_k(\rho)$ to be the unique point of intersection of $\partial E$
with $\partial B(0,\rho)$ within the sector $S$, namely
$\{A_k(\rho)\}=\partial E\cap\partial B(0,\rho)\cap S$. We now
introduce the angle $\theta_k(\rho)\in[0,\frac{2\pi}{k}]$ such that, in
polar coordinates, we can represent
$A_k(\rho)=(\rho\cos\theta_k(\rho),\rho\sin\theta_k(\rho))$ and  
\begin{eqnarray*}
\partial E \cap S=\bigcup_{\rho\in[\rho_1,\rho_2]}A_k(\rho)
\end{eqnarray*}
We now have  
\begin{eqnarray}\label{THC1:gamma_k}
\Hau^1(E\cap\partial B(0,\rho))=2k\rho\gamma_k(\rho).
\end{eqnarray}
By definition of the set $E^{as}$ we have that, for all
$\rho\in[\rho_1,\rho_2]$ 
\begin{eqnarray*}
\Hau^1(E\cap\partial B(0,\rho))=\Hau^1(E^{as}\cap\partial B(0,\rho))
\end{eqnarray*}
and moreover that
\begin{eqnarray}\label{TCH1:volumes}
|E^{as}\cap B|=|E\cap B|, \quad|E^{as}\cap (\R^2\setminus B)|=|E\cap
(\R^2\setminus B)|. 
\end{eqnarray}
By exploiting the same polar parameterization as before, we may
write 
\begin{eqnarray*}
\partial E^{as}=\bigcup_{\rho\in[\rho_1,\rho_2]}A_2(\rho)
\end{eqnarray*}
where $A_2(\rho)=(\rho\cos\theta_2(\rho),\rho\sin\theta_2(\rho))$ for
$\rho\in[\rho_1,\rho_2]$ and $\theta_2(\rho)$ is such that  
\begin{eqnarray}\label{THC1:gamma_2}
\Hau^1(E^{as}\cap\partial B(0,\rho))=4\rho\gamma_2(\rho).
\end{eqnarray}
Comparing \eqref{THC1:gamma_k} and \eqref{THC1:gamma_2} we obtain 
\begin{eqnarray}\label{TCH1:angoli}
\theta_2(\rho)={\frac k2}\theta_k(\rho)
\end{eqnarray}
We finally have
\begin{eqnarray}\label{eq:PEas-PE}
\nonumber
P(E^{as})&=&4\int_{\rho_1}^{\rho_2}\sqrt{1+\rho^2(\theta'_2)^2}\
d\rho=4\int_{\rho_1}^{\rho_2}\sqrt{1+\frac{k^2}{4}\rho^2(\theta'_k)^2}\
d\rho\\
&<& 2k\int_{\rho_1}^{\rho_2}\sqrt{1+\rho^2(\theta'_k)^2}\
d\rho=P(E),
\end{eqnarray}
where the last inequality follows since $k\geq 3$. 
To conclude the proof we show that the annular symmetrization
preserves the strict convexity of our sets, i.e., that for
$\Hau$-a.e.$x\in\partial E^{as}$ it holds  
\begin{equation}\label{TCH1:convexity}
H_{\partial E^{as}}(x)>0.
\end{equation}
In fact, were this the case, and arguing as in Step $1$, we would
immediately get ${\mathcal Z}(E^{as})=\{x_0\}$. Then, by the symmetry
of $E^{as}$, $x_0=0$ and finally, thanks to \eqref{TCH1:volumes}, we may
conclude that $\asym(E^{as})=\asym(E)$. 

To prove \eqref{TCH1:convexity} we observe that in $B$ and $\R^2
\setminus \overline B$ we have that $H_{\partial E}$ can be computed through the
parameterization $\theta_k=\theta_k(\rho)$ (with a slight abuse of
notation, we understand $H_{\partial E}$ as a function of $\rho$)
described before. More precisely, the well-known formula
\begin{eqnarray*}
0<H_{\partial
  E}(\rho)=-\frac{\rho^2(\theta_k')^3+\rho\theta_k''+2\theta_k'}{(1+(\rho\theta_k')^2)^{\frac
    32}}
\end{eqnarray*}
holds, and this turns out to imply that
$\rho^2(\theta_k')^3+\rho\theta_k''+2\theta_k'<0$. This last inequality,
together with \eqref{TCH1:angoli} and the fact that, by construction,
$\theta'_k(\rho)<0$, gives 
\begin{eqnarray*}
H_{\partial E^{as}}(\rho)=-\frac{k}{2}\cdot\frac{\frac
  {k^2}{4}\rho^2(\theta_2')^3+\rho\theta_2''+2\theta_2'}{(1+(\frac{k}{2}\rho\theta_2')^2)^{\frac
    32}}\geq -\frac{k}{2}\cdot \frac
{\rho^2(\theta_k')^3+\rho\theta_k''+2\theta_k'}{(1+(\frac{k}{2}\rho\theta_2')^2)^{\frac
    32}}>0. 
\end{eqnarray*} 
Thanks to \eqref{eq:PEas-PE} and by the definition of $\Qmj$ in
\eqref{def:Qmj}, we obtain that 
\begin{eqnarray}\label{symmetrization_step}
\Qmj(E)> \Qmj(E^{as}),
\end{eqnarray}
which is a contradiction with the minimality of $E = E^{(m)}_j$. This
concludes the proof of the theorem.
\end{proof}

\subsection*{Acknowledgements}
The research of the first author was partially supported by the European Research Council under FP7,
Advanced Grant n. 226234 ``Analytic Techniques for Geometric and Functional Inequalities" and by
the Deutsche Forschungsgemeinschaft through the Sonderforschungsbereich 611.
The second author thanks the Dipartimento di
Matematica e Applicazioni ``R. Caccioppoli'' of the Universit{\`a} degli Studi di Napoli \mbox{``Federico II''}, where he was visiting
scholar during the year 2009.


\bigskip

\bibliographystyle{siam}

\bibliography{bibisoperbc}

\end{document}